\documentclass[11pt,reqno,twoside]{article}

\usepackage{empheq}
\usepackage{graphicx}
\usepackage{float}
\usepackage{subfigure}
\usepackage{amssymb}
\usepackage{epstopdf}
\usepackage[asymmetric,top=3.5cm,bottom=4.3cm,left=3.1cm,right=3.1cm]{geometry}
\usepackage{hyperref}
\usepackage{bm}
\usepackage{amsmath,amsfonts,amsthm,mathrsfs,amssymb,cite}
\usepackage[usenames]{color}
\usepackage{array} 
\usepackage{paralist} 
\usepackage{verbatim} 
\usepackage{tabularx}
\usepackage{fancyhdr}
\usepackage{graphicx, subfigure,threeparttable}

\newtheorem{thm}{Theorem}[section]

\newtheorem{lem}{Lemma}[section]
\newtheorem{prop}{Proposition}[section]

\theoremstyle{definition}
\newtheorem{defn}{Definition}[section]
\theoremstyle{remark}

\newtheorem{rem}{Remark}[section]
\numberwithin{equation}{section}

\newcommand{\norm}[1]{\left\Vert#1\right\Vert}

\newcommand{\bu}{\mathbf{u}}

\newcommand{\bx}{\mathbf{x}}
\newcommand{\bvartheta}{\bm{\vartheta}}
\newcommand{\bvarkappa}{\bm{\varkappa}}
\newcommand{\balpha}{\bm{\alpha}}
\newcommand{\bnu}{\bm{\nu}}

\newcommand{\ba}{\mathbf{a}}

\newcommand{\be}{\mathbf{e}}

\newcommand{\bw}{\mathbf{w}}
\newcommand{\bv}{\mathbf{v}}
\newcommand{\bt}{\mathbf{t}}

\newcommand{\ri}{\mathrm{i}}

\newcommand{\abs}[1]{\left\vert#1\right\vert}

\newcommand{\bGa}{\mathbf{\Gamma}}
\newcommand{\bvarphi}{\bm{\varphi}}
\newcommand{\bpsi}{\bm{\psi}}
\newcommand{\bphi}{\bm{\phi}}
\newcommand{\by}{\mathbf{y}}

\newcommand{\bd}{\mathbf{d}}
\newcommand{\bS}{\mathbf{S}}
\newcommand{\bK}{\mathbf{K}}
\newcommand{\bI}{\mathbf{I}}
\newcommand{\bxi}{\bm{\xi}}
\newcommand{\bzeta}{\bm{\zeta}}

\newcommand{\Acal}{\mathcal{A}}

\newcommand{\Lcal}{\mathcal{L}}

\newcommand{\Hcal}{\mathcal{H}}

\newcommand{\Ocal}{\mathcal{O}}
\newcommand{\Tcal}{\mathcal{T}}

\renewcommand{\triangle}{\Delta}

\title{Mathematical Analysis of Subwavelength Resonances and Gradient Blow-up for Two Close-to-Touching Inclusions within the Two-Dimensional Elasticity}

\begin{document}

\author{
		Hongjie Li\footnote{Yau Mathematical Sciences Center, Tsinghua University, Beijing, China. The work of this author was substantially supported by NSFC grant (12401561). (hongjieli@tsinghua.edu.cn; hongjie\_li@yeah.net).}
		\and
		Longjuan Xu\footnote{Academy for Multidisciplinary Studies, Capital Normal University, Beijing 100048, China.
			The work of this author was partially supported by NSF of China (12301141) and Beijing Municipal Education Commission Science and Technology Project (KM202410028001). (longjuanxu@cnu.edu.cn).}
        \and
        Haolun Yang\footnote{Qiuzhen College, Tsinghua University, Beijing, China. (hl-yang24@mails.tsinghua.edu.cn).}
	}

\date{}
\maketitle

\begin{abstract}
Subwavelength elastic resonators can concentrate wave energy at length scales far below the incident wavelength, but their behavior becomes especially delicate when two resonators almost touch. In this paper, we give a rigorous analysis of a two-dimensional dimer made of two high-contrast hard inclusions embedded in a soft elastic matrix. The analysis confronts two features that are absent from the corresponding three-dimensional theory: the logarithmic low-frequency singularity of the two-dimensional elastic fundamental solution and the possible non-invertibility of the static single-layer potential. We overcome these difficulties by proving the invertibility of the correct frequency-dependent leading-order operator and then using it to reduce the resonance problem to a finite-dimensional system. For generally convex resonators satisfying natural symmetry assumptions, we derive six subwavelength resonant frequencies and identify their dependence on the material contrast $\delta$ and the inter-inclusion distance $\varepsilon$. We further quantify the resonant field concentration in the narrow gap. In the regime $\varepsilon=\Ocal(\delta^\beta)$, $0<\beta<2$, the gradients of the eigenmodes display sharply classified blow-up behavior: some modes attain the stronger rate $\Ocal(1/\varepsilon)$ at the closest point of the gap, while others blow up at the rate $\Ocal(1/\sqrt{\varepsilon})$ away from the centerline; the remaining mode is governed by a boundary mismatch mechanism. These results uncover resonance-induced singularities that are markedly stronger and more structured than those in static or non-resonant elasticity, and they provide a framework for analyzing larger clusters of closely spaced elastic subwavelength resonators.

\medskip 


\end{abstract}

\section{Introduction and problem formulation}

Subwavelength resonance is a central mechanism behind modern wave-control materials. When a small inclusion has physical parameters that contrast strongly with those of the surrounding medium, it can respond resonantly to waves whose wavelengths are much larger than the inclusion itself \cite{DHB6012, Liu00scie}. This apparently local effect can produce macroscopic consequences: a sparse collection of resonant inclusions may substantially alter the effective behavior of a composite \cite{CMP1, CMP2, KA7599,LRM4202}. Such phenomena have motivated a broad mathematical and physical literature, with applications ranging from invisibility cloaking \cite{ACK6055, MN1715, LLL9974, DLL262, LL6090} and super-resolution imaging \cite{BLLW9091, AZ0966, DLL1067} to super-absorption \cite{LSM0301, LL9023}. The rapid development of mechanical metamaterials and fabrication techniques \cite{QCJ2662} has made the quantitative analysis of these resonances increasingly important.

The best-known acoustic example is the Minnaert resonance of gas bubbles in a liquid \cite{M2277}, whose rigorous mathematical foundation has been established in \cite{AFGL007}. From an engineering viewpoint, however, controlling bubbles in liquid media can be difficult. A natural and robust alternative is to place resonant inclusions in soft elastic hosts \cite{CTL1646, LSM0301}, a setting whose acoustic-elastic theory has been developed in \cite{LLZ0572}. In purely elastic media, hard inclusions embedded in soft matrices form another important class of resonators, where resonance is driven by the large contrast in the Lam\'e parameters \cite{Liu00scie}. Rigorous studies of such elastic subwavelength structures include \cite{LX9769, LZ0855}. Related high-contrast resonance mechanisms also arise in electromagnetics, for instance in dielectric particles with large refractive index \cite{ALL4825, KKL121}.

Most existing mathematical analyses focus either on three-dimensional systems or on isolated resonators. Yet many experimentally relevant configurations, such as cylindrical inclusions, are effectively two-dimensional \cite{ABJH6625, CK3518, LL0165}. Moreover, resonators rarely act alone: when two inclusions are separated by a narrow gap, their interaction can shift the resonant frequencies and create intense field concentration in the region between them. This coupled-resonator effect is well known in physical studies of electrostatic and optical resonances \cite{MP7700, HE5116, KBS0116, RAB9988}. Mathematically, close-to-touching subwavelength resonators have been investigated for three-dimensional Helmholtz systems \cite{ADY2020, LZ2023}, for three-dimensional elastic systems \cite{LX9769}, and more recently for the two-dimensional Helmholtz equation \cite{DLX2025}. The two-dimensional elastic case, however, has remained largely open.

The purpose of this paper is to fill this gap by analyzing the subwavelength resonances of two closely spaced hard elastic inclusions in a two-dimensional soft matrix. This problem is physically rich because elastic waves couple compressional and shear components \cite{LLLW9004}. It is also mathematically delicate. Our approach is based on layer potential theory, with the elastic single-layer potential $\bS_D^\omega$ playing the central role. In contrast to the three-dimensional case \cite{LLL6014}, the two-dimensional fundamental solution has a logarithmic low-frequency expansion involving both $\omega^{2j}$ and $\omega^{2j}\ln\omega$ terms. Consequently, the leading operator is not the static single-layer potential $\bS_D^0$, but a frequency-dependent operator $\hat{\bS}_D^\omega$; see Lemma \ref{lem:assgl}. A further complication is that $\bS_D^0$ need not be invertible in two dimensions \cite{ak162}, whereas the corresponding three-dimensional static operator is invertible in the usual spaces \cite{DLL7678}. Establishing a workable invertibility theory is therefore a prerequisite for any precise resonance analysis.

The dimer geometry introduces an additional layer of structure. The relevant kernel has dimension six, rather than the two-dimensional kernel appearing in the corresponding two-dimensional Helmholtz problem. Thus, the resonant frequencies are encoded in a $6\times6$ interaction matrix. This enlarged modal space captures both translational and rotational elastic modes and allows the distance between the two inclusions to enter the leading-order frequency laws. As the gap width $\varepsilon$ tends to zero, the same coupling mechanism can force the eigenmode gradients to blow up in the narrow region between the resonators. Understanding which modes blow up, where they blow up, and at what rate is one of the main themes of the paper.

Our main contributions are as follows.

(1) We develop the two-dimensional potential-theoretic foundation needed for elastic subwavelength dimers. In particular, we prove that the leading-order operator $\hat{\bS}_D^\omega$ is invertible under a mild condition on the complex frequency $\omega$, even though the static operator $\bS_D^0$ may fail to be invertible. This result is stated in Theorem \ref{thm:insin} and clarified in Remark \ref{rem:ins}.

(2) We derive all six subwavelength resonant frequencies of the coupled elastic system. The resulting asymptotic formulas reveal three different frequency mechanisms: modes governed primarily by the high contrast $\delta$, modes determined through logarithmic nonlinear equations, and modes whose leading behavior depends simultaneously on $\delta$ and the gap distance $\varepsilon$. These formulas are given in Theorem \ref{thm:res} and further interpreted in Remark \ref{rem:res}.

(3) We establish sharp gradient estimates for the resonant eigenmodes as the inclusions approach each other. In the scaling regime $\varepsilon=\Ocal(\delta^\beta)$ with $0<\beta<2$, Theorem \ref{thm-eigonmode} shows that the modes separate into distinct blow-up profiles. Three modes reach the rate $\Ocal(1/\varepsilon)$ at the narrowest part of the gap; two modes remain bounded on the centerline but blow up like $\Ocal(1/\sqrt{\varepsilon})$ away from it; and one mode depends on a finer boundary mismatch. These resonance-driven singularities differ sharply from the classical static or non-resonant behavior, where the strongest two-dimensional blow-up is typically of order $\Ocal(1/\sqrt{\varepsilon})$ in electrostatics \cite{ABV2015, AKLLL2007, BLY2009, dfl2022, DL2019, KLY2015, KL2019, LY2009, P1989, HLL2026arxiv} and linear elasticity \cite{ACKLY, BLL2015, DLT3014, KY2019, LX2025}.

The remainder of this paper is organized as follows. Section \ref{sec-setup} presents the mathematical formulation of the problem. Section \ref{sec-pre} provides the necessary auxiliary results. Section \ref{sec-resonance} is devoted to the analysis of the resonant frequencies. Finally, Section \ref{sec:blowup} investigates the gradient blow-up estimates for the resonant waves within the narrow gap between adjacent resonators.

\section{Mathematical formulation}\label{sec-setup}
In this section, we first formulate the problem studied in this paper. 
We consider the configuration of a dimer consisting of two hard elastic inclusions embedded in a soft elastic material in two dimensions.
The soft elastic material is described by the Lam\'e parameters $(\lambda, \mu)$, satisfying the strong ellipticity conditions
\begin{equation*}
    \mu>0\quad  \text{and}\quad \lambda+\mu>0.
\end{equation*}
The density in the background is designated by $\rho$.
Let $D_{1}$ and $D_{2}$ denote the two hard inclusions, and they are two disjoint convex open subsets in $\mathbb R^2$ with $C^{2,\alpha}$ boundaries, $\alpha\in(0,1)$. The corresponding Lam\'e parameters and the density of the hard inclusions are parameterized by $(\tilde{\lambda}, \tilde{\mu})$ and $\tilde{\rho}$, respectively. 
Then, we introduce the three dimensionless contrast parameters $\delta$, $\eta$ and $\tau$ defined by 
\begin{equation}\label{eq:conpara}
(\tilde{\lambda}, \tilde{\mu}) = \frac{1}{\delta}(\lambda, \mu), \quad \eta = {\rho}/{\tilde{\rho}}, \quad \tau=\sqrt{\delta/\eta}.
\end{equation}
Since we are considering the configuration that hard inclusions are embedded within soft elastic materials (e.g. lead inclusions coated with silicone in \cite{Liu00scie, LZ0855}), the contrast parameters $\delta$, $\eta$ and $\tau$ satisfy the following conditions:
\[
 \delta\ll 1, \quad \eta\ll 1, \quad \mbox{and}\quad \tau\leq \Ocal(1).
\]

Define the domain $D$ by $D=D_1\cup D_2$ and denote by $D^{e}=\mathbb{R}^{2} \backslash \overline{D}$ the exterior of the domain $D$. 
Let $\mathbb{C}=(C_{ijkl})_{i,j,k,l=1}^2$ denote the elasticity tensor defined by 
$$C_{ijkl}=\lambda\delta_{ij}\delta_{kl} +\mu(\delta_{ik}\delta_{jl}+\delta_{il}\delta_{jk}).$$
Let $\bu^i$ be a time-harmonic incident elastic wave satisfying the elastic equation in the entire space $\mathbb{R}^2$
\begin{equation}\label{eq:inci}
\mathcal{L}_{ {\lambda}, {\mu}}\bu^i(\bx) + {\rho}\omega^2\bu^i(\bx) =0,
\end{equation}
where $\omega>0$ denotes the angular  frequency. 
In \eqref{eq:inci}, the Lam\'e operator $ \mathcal{L}_{\lambda, \mu}$ associated with the parameters $(\lambda,\mu)$ is defined by 
\begin{equation*}
 \Lcal_{\lambda,\mu}\bu=\nabla\cdot(\mathbb{C}e(\bu))=\mu \triangle\bu + (\lambda+ \mu)\nabla\nabla\cdot\bu.
\end{equation*}
In the last equation, $e(\bu)$ is the strain tensor given by
$$e(\bu)=\frac{1}{2}(\nabla \bu+(\nabla\bu)^{t}),$$
where $t$ signifies the transpose. It is well known that the elastic wave can be decomposed into the shear wave (s-wave) and the compressional wave (p-wave) \cite{KupTPET}, namely
\[
\bu = \bu_p + \bu_s,
\]
where $\bu_p$ and $\bu_s$ denote the p-wave and the s-wave, respectively. Moreover, the two kinds of waves satisfy the equations
\begin{equation*}
	\begin{split}
		(\triangle +k_p^2)\bu_p(\bx)  &= 0, \quad  \nabla\times \bu_p(\bx) = 0, \\
		(\triangle +k_s^2)\bu_s(\bx)  &= 0, \quad  \nabla\cdot \bu_s(\bx) = 0.
	\end{split}
\end{equation*}
Here $k_p$ and $k_s$ signify the wavenumbers of the p-wave and s-wave  respectively, and they are given by
\begin{equation}\label{pa:ksp}
	{k}_p=\omega/{c}_p,\quad {k}_s=\omega/{c}_s,
\end{equation}
with
\begin{equation*}
{c}_p=\sqrt{ ({\lambda} + 2 {\mu})/{\rho}},\quad {c}_s = \sqrt{{\mu}/{\rho}}.
\end{equation*}

Under an impinging wave $\bu^i$ given in \eqref{eq:inci}, the total displacement field ${\mathbf{u}}$ of the above described system is controlled by the following partial differential equations (PDEs) \cite{L0069}
\begin{equation}\label{eq:or2}
	\left\{
	\begin{array}{ll}
		\mathcal{L}_{ {\lambda},  {\mu}} \bu(\bx) + \rho\tau^2\omega^2 \bu(\bx) =0,    & \bx\in D, \medskip \\
		\mathcal{L}_{\lambda, \mu} \bu(\bx) + \rho\omega^2 \bu(\bx) = 0, & \bx\in D^{e},\medskip \\
		\bu(\bx)|_- =  \bu(\bx)|_+,      & \bx\in\partial D, \medskip \\
		\partial_{ {\bnu}} \bu(\bx)|_- = \delta\partial_{\bnu} \bu(\bx)|_+, &  \bx\in\partial D,
	\end{array}
	\right.
\end{equation}
where $\tau$ and $\delta$ are given in \eqref{eq:conpara}, and the subscripts $\pm$ indicate the limits from outside and inside of $D$, respectively.
In \eqref{eq:or2}, the traction operator $\partial_{\bnu}$ is defined by
\begin{equation*}
	\partial_{ {\bnu}}\bu(\bx)  =\lambda(\nabla \cdot \bu) \bnu +\mu\left(\nabla \bu+(\nabla \bu)^{t}\right) \bnu,
\end{equation*}
with $\bnu$ denoting the exterior unit normal vector to $\partial D$. 
In \eqref{eq:or2}, the scattering wave $\bu^s = \bu-\bu^i$ satisfies the following radiation condition \cite{KupTPET, LLZ2555}:
\begin{equation*}
	\begin{split}
		\partial_r\bu_p^s(\bx) - \ri k_p \bu_p^s(\bx)=&\mathcal{O}(|\bx|^{-3/2}),\\
		\partial_r\bu_s^s(\bx) - \ri k_s \bu_s^s(\bx)=&\mathcal{O}(|\bx|^{-3/2}),
	\end{split}
\end{equation*}
as $|\mathbf{x}|\rightarrow+\infty$, where $\ri$ signifies the imaginary unit. 

We would like to employ potential theory to analyze the resonant phenomenon of the system \eqref{eq:or2}. To that end, we first provide the potential theory for the two-dimensional elastic equation. The fundamental solution $\bGa^{\omega}=(\Gamma^{\omega}_{i,j})_{i,j=1}^2$ to the operator $\Lcal_{\lambda,\mu}+\rho \omega^2 $ in the two dimensions is given by \cite{ABJH6625}:
 \begin{equation}\label{eq:ef}
	\left(\Gamma_{i, j}^{\omega}\right)_{i, j=1}^{2}(\mathbf{x})=-\frac{\ri \bm{\delta}_{i j}}{4 \mu} H_{0}\left(k_{s}|\mathbf{x}|\right)+\frac{\mathrm{i}}{4 \rho \omega^{2}} \partial_{i} \partial_{j}\left(H_{0}\left(k_{p}|\mathbf{x}|\right)-H_{0}\left(k_{s}|\mathbf{x}|\right)\right),
 \end{equation}
where $H_0(\bx)$ is the Hankel function of the first kind of order $0$, and $k_p$ and $k_s$ are defined in \eqref{pa:ksp}. 
Then the single-layer potential associated with the fundamental solution $\bGa^{\omega}$ is defined by
\begin{equation*}\label{eq:single}
	\bS_{D}^{\omega}[\bvarphi](\bx)=\int_{\partial D} \bGa^{\omega}(\bx-\by)\bvarphi(\by)ds(\by), \quad \bx\in\mathbb{R}^2,
\end{equation*}
for $\bvarphi\in L^2(\partial D)^2$. On the boundary $\partial D$, the conormal derivative of the single-layer potential satisfies the following jump formula
\begin{equation}\label{eq:jump}
	\partial_{\bnu}   \bS_{ D}^{\omega}[\bvarphi]|_{\pm}(\bx)=\left( \pm\frac{1}{2}\bI +  \bK_{ D}^{\omega, *} \right)[\bvarphi](\bx), \quad \bx\in\partial D,
\end{equation}
where
\[
\bK_{ D}^{\omega, *} [\bvarphi](\bx)=\mbox{p.v.} \int_{\partial D} \partial_{\bnu_{\bx}} \bGa^{\omega}(\bx-\by)\bvarphi(\by)ds(\by), \quad \bx\in\partial D,
\]
with $\mbox{p.v.}$ standing for the Cauchy principal value. It is noted that the operator $ \bK_{ D}^{\omega, *}$ in \eqref{eq:jump} is called the Neumann-Poincar\'e (N-P) operator, which is a critical operator in the analysis of metamaterials. In what follows, we denote $\bS_{ D}^{0}$, $\bK_{ D}^{0, *}$ by $\bS_{ D}$, $\bK_{ D}^{ *}$, respectively, for simplicity.

With the help of the potential theory presented above, the solution to the system \eqref{eq:or2} can be written as 
\begin{equation}\label{eq:sol}
	\bu=
	\left\{
	\begin{array}{ll}
		{\bS}_{ D}^{\tau\omega}[\bvarphi](\bx), & \bx\in D,  \smallskip \\
		{\bS}_{ D}^{\omega}[\bpsi](\bx) +\bu^i, &  \bx\in \mathbb{R}^2\backslash \overline{D},
	\end{array}
	\right.
\end{equation}
for some density functions $\bvarphi, \bpsi \in L^2(\partial D)^2$.
By matching the transmission conditions on the boundary, i.e., the third and fourth conditions in \eqref{eq:or2} and with the help of the jump formula in \eqref{eq:jump}, the density functions $\bvarphi$ and $\bpsi$ in \eqref{eq:sol} satisfy the following system:
\begin{equation}\label{eq:or}
	\Acal(\omega,\delta) [\Phi](\bx)=F(\bx), \quad \bx\in\partial D,
\end{equation}
where
\[
\Acal(\omega,\delta)=  \left(
\begin{array}{cc}
	{\bS}_{ D}^{\tau\omega} &  -{\bS}_{ D}^{\omega}\medskip \\
	-\frac\bI{2} +  {\bK}_{ D}^{\tau\omega, *} & -\delta\left( \frac\bI{2} +  {\bK}_{ D}^{\omega, *} \right)\\
\end{array}
\right),
\;\;
\Phi= \left(
\begin{array}{c}
	\bvarphi \\
	\bpsi \\
\end{array}
\right),
\;\; \mbox{and} \;\; 
F= \left(
\begin{array}{c}
	\bu^i \\
	\delta\partial_{\bnu}  \bu^i \\
\end{array}
\right).
\]
For the subsequent discussion, we define the spaces $\Hcal=L^2(\partial D)^2\times L^2(\partial D)^2$ and $\Hcal^1=H^1(\partial D)^2\times L^2(\partial D)^2$. It is noted the operator $\Acal(\omega,\delta)$ is bounded from $\Hcal$ to $\Hcal^1$ (c.f. \cite{ABJH6625}). Next, we define the sub-wavelength resonance of the scattering system \eqref{eq:or2} based on the operator $\Acal(\omega,\delta)$.
\begin{defn}\label{def-frequency}
	The sub-wavelength resonance of the scattering system \eqref{eq:or2} occurs if there exists a frequency $|\omega|\ll 1$ with $\Re\omega> 0$ such that the operator $\Acal(\omega,\delta)$ has a nontrivial kernel, i.e., 
	\begin{equation*}
		\Acal(\omega,\delta)[\Phi](\bx)=0,
	\end{equation*}
	for some nontrivial $\Phi\in\Hcal$. Here, $\omega$ is called the resonant frequency (or eigenfrequency) and $\Phi$ is called the resonant eigenfunction. For each resonant frequency $\omega$, we define the corresponding resonant mode (or eigenmode) as 
 \begin{equation}\label{def-eigenmode}
     \bu=
	\left\{
	\begin{array}{ll}
		{\bS}_{ D}^{\tau\omega}[\bvarphi](\bx), & \bx\in D,  \smallskip \\
		{\bS}_{ D}^{\omega}[\bpsi](\bx) , &  \bx\in \mathbb{R}^2\backslash \overline{D}.
	\end{array}
	\right.
 \end{equation}
 The resonant mode here should be normalized in the sense that 
 \[
   \norm{\bu}_{L^2(\partial D)^2} =\Ocal(1). 
 \]
\end{defn}

In this paper, we shall systematically and comprehensively study the resonant phenomenon of the system \eqref{eq:or2}, including the resonant frequencies and the resonant eigenfunctions. Then, we also investigate the stress blow-up estimate of the scattering field between two hard inclusions within the resonant scenario. 

\section{Auxiliary results on the potential theories and some preliminaries}\label{sec-pre}

In this section, we present auxiliary results on the potential theories, especially the layer-potential operators. Then, we provide some necessary preliminary results for the subsequent analysis. 

Since we consider the resonant phenomenon in the subwavelength regime, we first provide some asymptotic analysis for the fundamental solutions and the related operators. The Hankel function $H_0^{(1)}(\omega|\bx|)$ has the following asymptotic expansion for $\omega\ll 1$ (c.f., \cite{DLX2025}),
\begin{equation}\label{eq:fuaym}
-\frac{\mathrm{i}}{4}H_0^{(1)}(\omega|\bx|)=\frac{1}{2\pi}\ln{|\bx|}+\eta_{\omega}+\sum_{j=1}^{\infty}\left(b_j\ln{\omega|\bx|}+c_j\right)\left(\omega|\bx|\right)^{2j},
\end{equation}
where 
\begin{equation}\label{def:eta}
\eta_{\omega} = \frac{1}{2\pi} (\ln \omega + \gamma - \ln 2) - \frac{\ri}{4}, \; b_j = \frac{(-1)^j}{2\pi} \frac{1}{2^{2j} (j!)^2}, \; c_j = b_j \left( \gamma - \ln 2 - \frac{\ri\pi}{2} - \sum_{n=1}^{j} \frac{1}{n} \right),
\end{equation}
with \(\gamma\) denoting the Euler constant. In particular,
\[
b_1 = -\frac{1}{8\pi}, \quad c_1 = -\frac{1}{8\pi} (\gamma - \ln 2 - 1 - \frac{\ri\pi}{2}).
\]
From the asymptotic expansion \eqref{eq:fuaym}, we obtain that the fundamental solution $\bGa^{\omega}$ defined in \eqref{eq:ef} has the following asymptotic expansion for $\omega\ll 1$ \cite{RG2025arxiv},
\begin{equation}\label{eq:gaaym}
	\bGa^{\omega}(\bx) = \bGa^{0,1}(\bx) +  \bGa^{0,2}(\bx) + \sum_{n=1}^{\infty
	}\left( \omega^{2n} \ln{\omega} \bGa^{n,1}(\bx) + \omega^{2n} \bGa^{n,2}(\bx) \right),
\end{equation}
where
\begin{equation}\label{def:ga01}
	\left(\Gamma_{i, j}^{0,1}\right)_{i, j=1}^{2}(\mathbf{x}) =\bm{\delta}_{i j} \hat{\eta}_{\omega},
\end{equation}
\begin{equation}\label{def:ga2}
	\left(\Gamma_{i, j}^{0,2}\right)_{i, j=1}^{2}(\mathbf{x}) =\bm{\delta}_{i j} \frac{1}{4\pi}\Big(\frac{1}{\mu}+\frac{1}{\lambda +2\mu}\Big)\ln\abs{\bx}  -\frac{1}{4\pi}\Big(\frac{1}{\mu}-\frac{1}{\lambda +2\mu}\Big) \frac{\bx_i\bx_j}{\abs{\bx}^2},
\end{equation}

\[
\begin{split}
	\left(\Gamma_{i, j}^{n,1}\right)_{i, j=1}^{2}(\mathbf{x}) = & \bm{\delta}_{i j} \abs{\bx}^{2n} 2b_{n+1}(n+1)\frac{1}{\rho} \Big( \frac{2n+3}{2n+2} \frac{1}{c_s^{2n+2}}  -\frac{1}{c_p^{2n+2}}   \Big) \\
	&+\frac{2\bx_i\bx_j}{\abs{\bx}^2\rho} \abs{\bx}^{2n}  2b_{n+1}n(n+1) \Big(\frac{1}{c_s^{2n+2}}-\frac{1}{c_p^{2n+2}}\Big),
\end{split}
\]

\[
\begin{split}
	\big(\Gamma_{i, j}^{n,2}\big)_{i, j=1}^{2}(\mathbf{x}) = &  \bm{\delta}_{i j} \frac{\abs{\bx}^{2n}}{\rho} \Big( \bigl(b_{n+1}+2c_{n+1}(n+1)\bigr)\big(\frac{1}{c_s^{2n+2}}-\frac{1}{c_p^{2n+2}}\big) + \frac{c_n}{c_s^{2n+2}} +   \\ 
	&  \;\; \frac{b_n}{c_s^{2n+2}} \ln\left({\abs{\bx}}/{c_s}\right)  +  2 b_{n+1}(n+1) \big(\frac{\ln(\abs{\bx}/c_s)}{c_s^{2n+2}}-\frac{\ln(\abs{\bx}/c_s)}{c_p^{2n+2}}\big)   \Big) + \\
	&\frac{2\bx_i\bx_j}{\abs{\bx}^2\rho} \abs{\bx}^{2n} \Big(((2n+1) b_{n+1}+2n(n+1)c_{n+1})\big(\frac{1}{c_s^{2j+2}}-\frac{1}{c_p^{2j+2}}\big) +  \\
	&  \qquad\qquad\qquad\qquad 2 b_{n+1}n(n+1) \big(\frac{\ln(\abs{\bx}/c_s)}{c_s^{2n+2}}-\frac{\ln(\abs{\bx}/c_s)}{c_p^{2n+2}}\big) \Big).
\end{split}
\]
In \eqref{def:ga01}, the parameter $\hat{\eta}_{\omega}$ is given by 
\begin{equation}\label{eq:defheta}
    \hat{\eta}_{\omega} = -\frac{\lambda+\mu -8\pi\left( 2c_1(\lambda+\mu) + \tilde{\eta}(\lambda+2\mu) \right)+2\mu\ln{k_p}-2(\lambda+2\mu)\ln{k_s}}{8\mu(\lambda+2\mu)\pi }, 
\end{equation}
with 
\begin{equation}\label{def:tildeeta}
	\tilde{\eta} = \frac{1}{2\pi} (\gamma - \ln 2) - \frac{\ri}{4}.
\end{equation}

With the help of the asymptotic expansion of the fundamental solution in \eqref{eq:gaaym}, the following lemma holds.
\begin{lem}\label{lem:assgl}
	The single layer potential operator $\bS_{D}^{\omega}$ has the following asymptotic expansion for $\omega\ll 1$,
	\begin{align*}
		\bS_D^\omega=\hat{\bS}_D^\omega+\omega^2\ln{\omega} \bS_{D,1}^{(1)}+\omega^2\bS_{D,1}^{(2)}+\Ocal(\omega^4\ln \omega),\\
        \bK_D^{\omega,*}=\bK_D^*+\omega^2\ln{\omega}\bK_{D,1}^{(1)}+\omega^2 \bK_{D,1}^{(2)} +\Ocal(\omega^2\ln{\omega}),
	\end{align*}
	where
	\[
		\hat{\bS}_D^{\omega} [\bvarphi]=\int_{\partial D}(\bGa^{0,1}+ \bGa^{0,2})(\bx-\by)\bvarphi(\by)ds(\by),\quad \bS_D[\bvarphi]=\int_{\partial D} \bGa^{0,2}(\bx-\by)\bvarphi(\by)ds(\by), 
	\]
	\[
		\bS_{D,1}^{(j)}[\bvarphi]=\int_{\partial D} \bGa^{1,j}(\bx-\by)\bvarphi(\by)ds(\by), \quad \bK_{D}^*[\bvarphi]=\int_{\partial D} \partial_{\bnu_{\bx}}\bGa^{0,2}(\bx-\by)\bvarphi(\by)ds(\by),
	\]
	and
	\[
		\bK_{D,1}^{(j)}[\bvarphi]=\int_{\partial D} \partial_{\bnu_{\bx}}\bGa^{1,j}(\bx-\by)\bvarphi(\by)ds(\by).
	\]
\end{lem}

For the further analysis, we first need to study the invertibility of the operator $\hat{\bS}_D^{\omega}.$ It is noted  that $\bS_D$ may have a nontrivial kernel. Nevertheless, if density functions satisfy certain conditions, the operator $\bS_D$ is injective \cite{ABJH6625}.
\begin{lem}\label{lem:kse0}
	For any $\bvarphi \in L^2(\partial D)^2$ with $\int_{\partial D}\bvarphi =0,$ if there holds $\bS_D[\bvarphi]=0,$ then $\bvarphi =0.$
\end{lem}

\begin{lem}\label{lem:dimsd}
	The dimension of the kernel of the operator $\bS_D$ is at most $2$ , that is $$\dim{\ker{\bS_D}}\leq 2.$$
\end{lem}
\begin{proof}
	Suppose that three functions $\bvarphi_1, \bvarphi_2, \bvarphi_3\in \ker \bS_D$.
	We first assume that $\int_{\partial D} \bvarphi_1 \neq \int_{\partial D} \bvarphi_2 $. Then we can find two constants $c_1$ and $c_2$ such that
	\[
	 \int_{\partial D} \bvarphi_3 = c_1 \int_{\partial D} \bvarphi_1 + c_2 \int_{\partial D} \bvarphi_2.
	\]
	We define a new function 
    \[
    \bvarphi = \bvarphi_3 - c_1 \bvarphi_1 - c_2 \bvarphi_2.
    \]
    It is obvious that $\bS_D[\bvarphi]=0$ and $\int_{\partial D} \bvarphi=0$. From Lemma \ref{lem:kse0}
, we can obtain $\bvarphi=0$.
If $\int_{\partial D} \bvarphi_1 = \int_{\partial D} \bvarphi_2 $, following the same argument, we can obtain that $\bvarphi_1 = \bvarphi_2 $.
The proof is completed. 
\end{proof}

Even though the operator $\bS_D$ may have a nontrivial kernel, the operator $\Tcal$ described in the following lemma has a bounded inverse \cite{VM3767}.
\begin{lem}\label{lem:ings}
	The operator $\Tcal: L^2(\partial D)^2 \times \mathbb{R}^2 \rightarrow H^1(\partial D)^2 \times \mathbb{R}^2$ defined by
			\[
			\Tcal[\bvarphi, \bt] = \left( \bS_D[\bvarphi] + \bt, \; \int_{\partial D}\bvarphi(\by)ds(\by) \right)
			\]
			has a bounded inverse.
\end{lem}

Next, we examine the invertibility of the operator $\hat{\bS}_D^{\omega}$, which is established in the following theorem.
\begin{thm}\label{thm:insin}
	For the parameter $ \hat{\eta}_{\omega}$ given in \eqref{eq:defheta}, if  ${\omega}\in \mathbb{C}$ is chosen such that
    \[
    \Im \hat{\eta}_{\omega}\neq 0, \quad \Re{ \hat{\eta}_{\omega}}\neq -\frac 1 2\frac{\tilde{b}_{11} + \tilde{b}_{22}}{\tilde{b}_{11}\tilde{b}_{22} - \tilde{b}_{12}\tilde{b}_{21}},
    \]
    where $\tilde{b}_{ij}$ is given in \eqref{eq:detb},
    then the operator $\hat{\bS}_D^{\omega}$ is invertible from $L^2(\partial D)^2$ to $H^1(\partial D)^2$.
\end{thm}
\begin{rem}\label{rem:ins}
    Before giving the proof of the theorem, we explain the conditions required in Theorem \ref{thm:insin}. The condition $\Im \hat{\eta}_{\omega}\neq 0$ is equivalent to 
    \[
    \Im\ln {\omega}\neq \frac{\lambda+2\mu}{\lambda+\mu}\pi, \quad 
    \]
    which requires that $\omega$ is not located on a ray  in the complex plane. The condition $\Re{ \hat{\eta}_{\omega}}\neq -\frac 1 2\frac{\tilde{b}_{11} + \tilde{b}_{22}}{\tilde{b}_{11}\tilde{b}_{22} - \tilde{b}_{12}\tilde{b}_{21}}$ is equivalent to 
    \[
    \ln{|\omega|} \neq \frac{ (1-16\pi c_1) (\lambda+\mu)  - 2\mu \ln{c_p} + 2(\lambda + 2\mu)\big(\ln{c_s} - 2(\gamma-\ln 2) +  \frac{4\mu\pi(\tilde{b}_{11} + \tilde{b}_{22})}{\tilde{b}_{11}\tilde{b}_{22} - \tilde{b}_{12}\tilde{b}_{21}}\big) }{2(\lambda + \mu)},
    \]
    which requires that $\omega$ is not located on a circle in the complex plane. Therefore, as long as ${\omega}\in \mathbb{C}$ is not located on a ray and a circle in the complex plane, the operator $\hat{\bS}_D^{\omega}$ is invertible from $L^2(\partial D)^2$ to $H^1(\partial D)^2$.
\end{rem}

\begin{proof}[Proof of Theorem \ref{thm:insin}]
	First, it is noted that the operator $\hat{\bS}_D^{\omega}$ is a Fredholm operator \cite{VM3767}. Thus, we just need to study the kernel of the operator $\hat{\bS}_D^{\omega}$. Suppose that there exists a nontrivial function $\bvarphi \in L^2(\partial D)^2$ such that 
	\begin{equation}\label{eq:kset}
	\hat{\bS}_D^{\omega}[\bvarphi]=\bS_D[\bvarphi]+\hat{\eta}_{\omega} \int_{\partial D}\bvarphi =0.
	\end{equation}
    
    $\bullet$ We first consider the case that the operator $\bS_D$ has a nontrivial kernel. From Lemma \ref{lem:dimsd}, we have that the dimension of the kernel of the operator $\bS_D$ is at most $2$.  Thus, we consider the following two cases.
	\begin{enumerate}
	   \item [1)] If the dimension of the kernel of the operator $\bS_D$ is $2$, then we can find two functions $\bphi_1, \bphi_2 \in \ker \bS_D$ such that $\int_{\partial D} \bphi_1 \neq \int_{\partial D} \bphi_2 $. Multiplying $\bphi_i$, $i=1,2$, on both sides of \eqref{eq:kset} and using the fact that the operator $\bS_D$ is self-adjoint, we have that
	   \begin{equation}\label{eq:kset2}
	    \hat{\eta}_{\omega} \int_{\partial D}  \bvarphi(y) \cdot \int_{\partial D}\bphi_i(y) ds(y) =0, \quad i=1,2.
	   \end{equation}
	   Then from the choice of $\omega$ and $\int_{\partial D} \bphi_1 \neq \int_{\partial D} \bphi_2 $, we obtain $\int_{\partial D} \bvarphi =0$. This together with Lemma \ref{lem:kse0} and the equation \eqref{eq:kset} yields $\bvarphi =0$.
	   
	   \item [2)] If the dimension of the kernel of the operator $\bS_D$ is $1$, then we can find a function $\bphi \in \ker \bS_D$ with $\int_{\partial D} \bphi \neq 0$. Similar to equation \eqref{eq:kset2}, we have that
	   \begin{equation}\label{eq:kset3}
	    \hat{\eta}_{\omega} \int_{\partial D}  \bvarphi(y) \cdot \int_{\partial D}\bphi(y) ds(y) =0.
	   \end{equation}
	    Next, we show that the function $\bvarphi$ is a real function. In fact, if the function $\bvarphi$ is a complex function, then we can write $\bvarphi = \bvarphi_1 + \mathrm{i} \bvarphi_2$ with $\bvarphi_1, \bvarphi_2$ being real functions. From equation \eqref{eq:kset3}, we obtain that 
	   \begin{equation}\label{eq:kset4}
	    \int_{\partial D}  \bvarphi_2(y) =s_1 \ba,
	   \end{equation}
	   where $\ba= \int_{\partial D}  \bvarphi_1(y)$.
	   Without loss of generality, we assume that $\int_{\partial D}  \bvarphi_1 \neq 0$ and $s_1\in \mathbb{R}$. We first consider the case that $s_1\neq 0$. Substituting the last equation \eqref{eq:kset4} into equation \eqref{eq:kset}, we have that
	   \[
	   \bS_D[\bvarphi_1] + \mathrm{i} \bS_D[\bvarphi_2] + (\Re\hat{\eta}_{\omega} + \mathrm{i} \Im\hat{\eta}_{\omega}) (1 + \mathrm{i} s_1) \int_{\partial D} \bvarphi_1 =0.
	   \]
	   Taking the real part and the imaginary part of the last equation, we have that
	   \[
       \begin{cases}
	   \bS_D[\bvarphi_1] + (\Re\hat{\eta}_{\omega} - s_1 \Im\hat{\eta}_{\omega}) \int_{\partial D} \bvarphi_1 =0,
	  \medskip \\
	   \bS_D[\bvarphi_2] + (s_1 \Re\hat{\eta}_{\omega} + \Im\hat{\eta}_{\omega} ) \int_{\partial D} \bvarphi_1 =0.
       \end{cases}
	   \]
	   Then we consider the following two systems with $s_1\neq 0$,
	   \[
		\begin{cases}
			\bS_D[s_1 \bvarphi_1] + (s_1\Re\hat{\eta}_{\omega} - s^2_1 \Im\hat{\eta}_{\omega}) \int_{\partial D} \bvarphi_1 =0,  \medskip \\
			\int_{\partial D}  s_1\bvarphi_1(y) = s_1\ba,
		\end{cases}
		\]
	 and
	   \[
		\begin{cases}
			\bS_D[\bvarphi_2] + (s_1 \Re\hat{\eta}_{\omega} + \Im\hat{\eta}_{\omega} ) \int_{\partial D} \bvarphi_1 =0,  \medskip \\
			\int_{\partial D}  \bvarphi_2(y) = s_1\ba.
		\end{cases}
		\]
		From Lemma \ref{lem:ings}, we have that the last two systems have the same solution. Thus, we obtain that $s_1^2 = -1$. This is a contradiction with the fact that $s_1\in \mathbb{R}$. 

		Now we consider the case for $s_1 = 0$. If $s_1 = 0$, then we construct a new function 
		\[
		\bvarphi_3 = \bvarphi_1 - \frac{\Re \hat{\eta}_{\omega}}{\Im \hat{\eta}_{\omega}} \bvarphi_2.
		\]
		We can easily verify that $\bS_D[\bvarphi_3]=0$ and $\int_{\partial D} \bvarphi_3 =\ba\neq 0$. This is a contradiction with our assumption that the dimension of the kernel of the operator $\bS_D$ is $1$. Finally, we have that $\bvarphi$ is a real function. Again, from the choice of $\omega$, we obtain that $\int_{\partial D} \bvarphi =0$. Then we further conclude that $\bvarphi=0$.
	\end{enumerate}

$\bullet$ Next, we consider the case that the operator $\bS_D$ is invertible. From \eqref{eq:kset}, there holds that
\begin{equation}\label{eq:scon}
    \bS_D[\bvarphi] =\ba_1,
\end{equation}
where $\ba_1=(\ba_{1,1},\ba_{1,2})^t= -\hat{\eta}_{\omega} \int_{\partial D} \bvarphi$. Let $\be_1=(1,0)^t$ and $\be_2=(0,1)^t$.
Then, we define two functions 
\[
\bpsi_1=\bS_D^{-1}[\be_1], \quad \bpsi_2=\bS_D^{-1}[\be_2],
\]
since the operator $\bS_D$ is invertible. It follows from \eqref{eq:scon} that the function $\bvarphi$ can be written as
\[
\bvarphi=\ba_{1,1}\bpsi_1 + \ba_{1,2}\bpsi_2.
\]
Substituting the last equation into \eqref{eq:kset} gives that 
\begin{equation}\label{eq:pin1}
    \ba_{1,1}\big(\be_1 + \hat{\eta}_{\omega} \int_{\partial D}\bpsi_1 \big) + \ba_{1,2} \big(\be_2 + \hat{\eta}_{\omega} \int_{\partial D}\bpsi_2\big)=0.
\end{equation}
 We rewrite the identity \eqref{eq:pin1} into the following matrix form 
    \[
     \widetilde{B}\ba_{1}=\mathbf{0},
    \]
    where 
    \begin{equation}\label{eq:detb}
    \widetilde{B} = \begin{pmatrix}
		1 + \hat{\eta}_{\omega} \tilde{b}_{11} & \hat{\eta}_{\omega} \tilde{b}_{21}\\
		\hat{\eta}_{\omega} \tilde{b}_{12}& 1 + \hat{\eta}_{\omega} \tilde{b}_{22} 
	\end{pmatrix}, \quad \tilde{b}_{ij} = \int_{\partial D}\bpsi_{i,j}, \;\; i,j=1,2.  
    \end{equation}
    Direct calculation shows that 
    \[
    \Im\det(\widetilde{B})= \Im{\hat{\eta}_{\omega}} \big( (\tilde{b}_{11} + \tilde{b}_{22} ) + 2 \Re{\hat{\eta}_{\omega}} (\tilde{b}_{11}\tilde{b}_{22} - \tilde{b}_{12}\tilde{b}_{21}) \big).
    \]
    Then, from the choice of $\omega$, i.e., $\Im{\hat{\eta}_{\omega}}\neq 0$ and $\Re{\hat{\eta}_{\omega}} \neq -\frac 1 2\frac{\tilde{b}_{11} + \tilde{b}_{22}}{\tilde{b}_{11}\tilde{b}_{22} - \tilde{b}_{12}\tilde{b}_{21}}$, we have that $\det(\widetilde{B})\neq 0$. Thus, we have $\ba_1=0$. Finally, we can directly conclude $\bvarphi=0$ from \eqref{eq:scon} due to the invertibility of the operator $\bS_D$. 

    The proof is completed.
\end{proof}

Next, we consider the Neumann boundary value problem:
\[
\begin{cases}
	\Lcal_{\lambda,\mu} \bu =0,\quad & \bx\in D,\\
	\partial_{\bnu}\bu =0,\quad & \bx\in \partial D.
\end{cases}
\]
Denote $\mho$ the space spanned by the solution of the last equation, which is given by  
\[
\mho =\{\mathbf{a}+\mathbf{B}\mathbf{x},~\mathbf{a}\in \mathbb{R}^2, ~\mathbf{B}\in M^2\},
\]
where $M^2$ is the space of antisymmetric matrices. Straightforward calculations show that the space $\mho$ is spanned by the following functions:
\[
{\tilde{{\bm{\vartheta}}}}_i=\begin{cases}
	\bvarkappa_i \quad &\text{in}\  D_1\\
	0\quad &\text{in}\   D_2
\end{cases},  \quad {\tilde{{\bm{\vartheta}}}}_{i+3}=\begin{cases}
	0 \quad &\text{in}\  D_1\\
	\bvarkappa_i \quad &\text{in}\  D_2\\
\end{cases}, \quad 1\leq i\leq 3,
\]
where 
\begin{align}\label{eq:dpsi}
	\bvarkappa_1=\begin{bmatrix}
	1\\
	0
	\end{bmatrix},\quad \bvarkappa_2=\begin{bmatrix}
		0\\
		1
	\end{bmatrix},\quad \bvarkappa_3=\begin{bmatrix}
		\bx_2\\
		-\bx_1
	\end{bmatrix}.
	\end{align}
Let $\bm{\vartheta}_i$, $1\leq i\leq 6$, denote the trace of the function $\tilde{\bm{\vartheta}}_i$ on $\partial D$.

\begin{lem}\label{lem:k0}\cite{ABJH6625}
	The kernel of the operator $-{\bf I}/2+\bK_D$ coincides with the space $\mho $, where $\bK_D$ is the adjoint operator $\bK_D^*.$
\end{lem}

For the analysis of the resonant phenomenon, we need to investigate the following functions 
\begin{equation}\label{def:xi}
\bxi_i=\left(\hat{\bS}_D^{\tau\omega}\right)^{-1}[{\bm{\vartheta}}_i],\quad \bzeta_i =\left(\hat{\bS}_D^\omega\right)^{-1}[{\bm{\vartheta}}_i], \quad 1\leq i\leq 6.
\end{equation}
By the jump formula in \eqref{eq:jump} and the asymptotic expansion of ${\bS}_D^\omega$ in Lemma \ref{lem:assgl}, we can easily verify the following lemma.
\begin{lem}\label{lem:ks0}\cite{ABJH6625}
The functions $\bxi_i$ and $\bzeta_i,1\leq i\leq 6,$ all belong to the kernel of the operator $-\frac{1}{2}{\bf I} +\bK_D^*.$ 
\end{lem}

\begin{lem}\label{lem:ks01}
	For any $\bvarphi \in L^2(\partial D)^2$, there holds that
	\[
		\int_{\partial D} \bK_{D,1}^{(1)}[\bvarphi] \cdot {\bm{\vartheta}_i}  =\tilde{\alpha}\rho \int_D\tilde{\bm{\vartheta}_i} \cdot\int_{\partial D}\bvarphi,
	\]
	for $1\leq i\leq 6$, where 
	\[
 \tilde{\alpha}=\frac{\lambda+\mu}{4\pi(\lambda+2\mu)}.
\]
\end{lem}

\begin{proof}
By the definition of the operator $\bK_{D,1}^{(1)}$ in Lemma \ref{lem:assgl}, Green's identity and the fact $\partial_{\bnu_{\bx}}\tilde{\bm{\vartheta}_i}=0$, we have that
 \begin{align*}
	\int_{\partial D} \bK_{D,1}^{(1)}[\bvarphi](\bx)\cdot {\bm{\vartheta}_i}(\bx)\ \mathrm{d}s(\bx)&=\int_{\partial D}\int_{\partial D} \partial_{\bnu_{\bx}}\bGa^{1,1}(\bx-\by)\bvarphi(\by)\ \mathrm{d}s(\by)\cdot {\bm{\vartheta}_i}(\bx)\ \mathrm{d}s(\bx)\\
 &= \int_{\partial D}\int_{ D} \Lcal_{\lambda,\mu}\bGa^{1,1}(\bx-\by)\cdot \tilde{\bm{\vartheta}_i} (\bx)\ \mathrm{d}\bx \cdot \bvarphi(\by)\ \mathrm{d}s(\by) \\
  &=\tilde{\alpha}\rho \int_D\tilde{\bm{\vartheta}_i}(\bx) \ \mathrm{d}\bx \cdot \int_{\partial D}\bvarphi(\by)\ \mathrm{d}s(\by).
\end{align*}
The proof is completed.
\end{proof}

\begin{lem}\label{lem:ks02}
	For any $\bvarphi \in L^2(\partial D)^2$, there holds that
	\[
		\int_{\partial D} \bK_{D,1}^{(2)}[\bvarphi]\cdot {\bm{\vartheta}_i}  = -\rho\int_D{ \hat{\bS}_D^\omega[\bvarphi] \cdot \tilde{\bm{\vartheta}_i}} - \tilde{\alpha}\rho \ln\omega \int_D \tilde{\bm{\vartheta}_i}  \cdot \int_{\partial D}\bvarphi ,
	\]
	for $1\leq i\leq 6$, where $\tilde{\alpha}$ is given in Lemma \ref{lem:ks01}.
\end{lem}

\begin{proof}
	By the definition of the operator $\bK_{D,1}^{(2)}$ in Lemma \ref{lem:assgl}, Green's identity and the fact $\partial_{\bnu_{\bx}}\tilde{\bm{\vartheta}_i}=0$, we have that
 \begin{align*}
	\int_{\partial D} \bK_{D,1}^{(2)}[\bvarphi](\bx)\cdot {\bm{\vartheta}_i}(\bx)\ \mathrm{d}s(\bx)&=\int_{\partial D}\int_{\partial D} \partial_{\bnu_{\bx}}\bGa^{1,2}(\bx-\by)\bvarphi(\by)\ \mathrm{d}s(\by)\cdot {\bm{\vartheta}_i}(\bx)\ \mathrm{d}s(\bx)\\
 &= \int_{\partial D}\int_{ D} \Lcal_{\lambda,\mu}\bGa^{1,2}(\bx-\by)\cdot \tilde{\bm{\vartheta}_i} (\bx)\ \mathrm{d}\bx \cdot \bvarphi(\by)\ \mathrm{d}s(\by) \\
  &= -\rho\int_D{ \hat{\bS}_D^\omega[\bvarphi] \cdot \tilde{\bm{\vartheta}_i}} - \tilde{\alpha}\rho \ln\omega \int_D\tilde{\bm{\vartheta}_i}  \cdot \int_{\partial D}\bvarphi  .
\end{align*}
The proof is completed.

\end{proof}

\section{Resonant analysis}\label{sec-resonance}
In this section, we investigate the resonant phenomena of the system \eqref{eq:or2}, including the resonant frequencies and the resonant eigenfunctions. To obtain an explicit expression of the resonant frequencies and eigenfunctions, we assume that the domain $D$ is symmetric with respect to the origin and the $\bx_1$-axis in this section.

Denote by 
\begin{equation}\label{def:alpha}
	\balpha_i=\int_{\partial D} \bzeta_i(\by)  , \quad \balpha_{ij}=\int_{\partial D_j} \bzeta_i(\by) , \quad \beta_{ij}=\int_{\partial D_j} \bzeta_i(\by)\cdot \begin{pmatrix}
		\by_2\\
		-\by_1
	\end{pmatrix},
\end{equation}
for $1\leq i\leq 6$ and $j=1,2$.
In what follows, we use the notation $\balpha_{i,n}$ to denote the $n$-th component of the vector $\balpha_i$ and the same holds for the vectors $\balpha_{ij}$.

\begin{lem}\label{lem:symor}
	For the parameters $\balpha_i$ and $\balpha_{ij}$, $1\leq i\leq 6$ and $j=1,2$, defined in \eqref{def:alpha}, if the domain $D$ is symmetric with respect to the origin, then we have that
	\begin{equation}
	    \label{eq1.38}
	    \begin{aligned}
		\balpha_{11}=\balpha_{42}, \quad \balpha_{12} =\balpha_{41}, \quad \balpha_{31}=-\balpha_{62},\\
		\balpha_{21}=\balpha_{52}, \quad \balpha_{22}=\balpha_{51}, \quad \balpha_{32}=-\balpha_{61},
	\end{aligned}
    \end{equation}
	and
	\begin{equation}
	    \label{eq1.39}
        	\balpha_1=\balpha_4,\quad \balpha_{2}=\balpha_{5},\quad \balpha_{3}=-\balpha_{6}.
	\end{equation} 
\end{lem}
\begin{proof}
	With the help of the symmetry of the domain $D$ and the change of variable, we have that 
	\[
	\begin{split}
		\bvartheta_1 = \hat{\bS}_D^\omega[\bzeta_1](\bx) & =   \int_{\partial D}(\bGa^{0,1}+ \bGa^{0,2})(\bx-\by)\bzeta_1(\by)ds(\by) \\
		& \overset{\by\to -\by}{=}   \int_{\partial D}(\bGa^{0,1}+ \bGa^{0,2})(\bx+\by)\bzeta_1(-\by)ds(\by) \\
		& \overset{\bx=-\widetilde{\bx}}{=}   \int_{\partial D}(\bGa^{0,1}+ \bGa^{0,2})(\widetilde{\bx}-\by)\bzeta_1(-\by)ds(\by).
	\end{split}
	\]
	Thus, we can conclude that $\bzeta_4(\by) = \bzeta_1(-\by)$. Following the same argument, we can obtain that $\bzeta_5(\by) = \bzeta_2(-\by)$ and $\bzeta_6(\by) = -\bzeta_3(-\by)$. Then, direct calculation shows that
	\[
	\balpha_{11}=\int_{\partial D_1} \bzeta_1(\by)\ \mathrm{d}s(\by)=\int_{\partial D_2} \bzeta_1(-\by)\ \mathrm{d}s(\by)=\int_{\partial D_2} \bzeta_4(\by)\ \mathrm{d}s(\by) =\balpha_{42}.
\]
Similarly, we can obtain the other equalities in \eqref{eq1.38}.
By the definition of $\balpha_{i}$ in \eqref{def:alpha}, we get \eqref{eq1.39}.
\end{proof}

\begin{lem}\label{lem:symx1}
	If the domain $D$ is symmetric with respect to the $\bx_1$-axis, then we have that
	\[
		\bzeta_4 = -\widetilde{\bzeta}_1, \quad \bzeta_5 = \widetilde{\bzeta}_2, \quad \bzeta_6 = \widetilde{\bzeta}_3,
	\]
	where $\widetilde{\bzeta}_i$, $i=1,2,3$, are defined by
	\begin{equation}\label{eq:symx2}
		\widetilde{\bzeta}_i(\by_1, \by_2) = \begin{pmatrix}
		-\bzeta_{i,1}(\by_1, -\by_2)\\
		\bzeta_{i,2}(\by_1, -\by_2)
	\end{pmatrix}.
	\end{equation}
\end{lem}
\begin{proof}
	Since the domain $D$ is symmetric with respect to the $\bx_1$-axis, by the change of variables, we have that
	\begin{equation}\label{eq:symx1}
	\begin{split}
		\bvartheta_2 = \hat{\bS}_D^\omega[\bzeta_2](\bx) & =   \int_{\partial D}(\bGa^{0,1}+ \bGa^{0,2})(\bx_1-\by_1,\bx_2-\by_2)\bzeta_2(\by_1, \by_2)ds(\by) \\
			& \overset{\by_2\to -\by_2}{=}   \int_{\partial D}(\bGa^{0,1}+ \bGa^{0,2})(\bx_1-\by_1,\bx_2+\by_2)\bzeta_2(\by_1, -\by_2)ds(\by) \\
			& \overset{\bx_2=-\widetilde{\bx}_2}{=}   \int_{\partial D}(\bGa^{0,1}+ \bGa^{0,2})(\bx_1-\by_1,-\widetilde{\bx}_2+\by_2)\bzeta_2(\by_1, -\by_2)ds(\by)\\
			& = \int_{\partial D}(\widetilde{\bGa}^{0,1}+ \widetilde{\bGa}^{0,2})(\bx_1-\by_1,\widetilde{\bx}_2-\by_2)\widetilde{\bzeta}_2(\by_1, \by_2)ds(\by) \\
			& =   \int_{\partial D}(\bGa^{0,1}+ \bGa^{0,2})(\bx_1-\by_1,\widetilde{\bx}_2-\by_2)\widetilde{\bzeta}_2(\by_1, \by_2)ds(\by),
	\end{split}
	\end{equation}
where $\widetilde{\bzeta}_2$ is given in \eqref{eq:symx2} and 
\[
	\left(\widetilde{\bGa}^{0,m}_{i, j}\right)_{i, j=1}^{2} = (-1)^{i}\left({\bGa}^{0,m}_{i, j}\right)_{i, j=1}^{2}, \quad m=1,2.
\]
The last identity in \eqref{eq:symx1} follows from the fact that the first component of $\bvartheta_2$ is zero. Thus, we have that $\bzeta_5 = \widetilde{\bzeta}_2$.
Following the same argument, we obtain that
\[
	\bzeta_6 = \widetilde{\bzeta}_3, \quad \bzeta_4 = -\widetilde{\bzeta}_1,
\]
with $\widetilde{\bzeta}_3$ and $\widetilde{\bzeta}_1$ given in \eqref{eq:symx2}.
\end{proof}

\begin{prop}\label{prop:A}
Define the matrix $A$ by 
	\begin{equation}\label{def-Aij}
	    A_{ij}=\int_{\partial D} {\bm{\bvartheta}}_i(\by) \cdot \bzeta_j(\by) \ \mathrm{d}s(\by),\quad i,j=1,\dots,6.
	\end{equation}
	If the domain $D$ is symmetric with respect to the origin and the $\bx_1$-axis, then the components of the matrix $A$ satisfy the following properties:
	\begin{equation}
	    \label{eq2.04}
        		A_{i2}=A_{i5}=A_{2i}=A_{5i}=0, \quad i=1,3,4,6,
	\end{equation}
	\[
		A_{ii}=A_{(i+3)(i+3)}=\balpha_{i1,i}, \quad A_{i(i+3)}=A_{(i+3)i}=\balpha_{i2,i} \quad i=1,2,
	\]
	\[
       A_{13}=A_{31}=-A_{64}=-A_{46}=\balpha_{31,1},  \quad A_{33}=A_{66}=\beta_{31},
	\]
	\[
       A_{16}=A_{61}=-A_{43}=-A_{34}=-\balpha_{32,1},  \quad A_{36}=A_{63}=\beta_{32}.
	\]
\end{prop}
\begin{proof}
	It is noted that the matrix $A$ is symmetric, which follows from 
	\[
		A_{ij}=\int_{\partial D} {\bm{\bvartheta}}_i \cdot \bzeta_j  = \int_{\partial D}  \hat{\bS}_D^\omega [\bzeta_i] \cdot \bzeta_j =\int_{\partial D} \bzeta_i \cdot \hat{\bS}_D^\omega [\bzeta_j]  = \int_{\partial D} {\bm{\bvartheta}}_j \cdot \bzeta_i=A_{ji}.
	\]
	From Lemma \ref{lem:symor}, direct calculation shows that
	\[
		A_{ij}=A_{i+3,j+3}, \quad 1\leq i,j\leq 2, \quad A_{33}=A_{66}, 
	\]
	\[
		A_{ij}=A_{(i+3)(j-3)}, \quad i=1,2,\; j=4,5, \quad A_{36}=A_{63},
	\]
	\[
	 A_{i3}= - A_{(i+3)6}, \quad A_{3i}= - A_{6(i+3)}, \quad i=1,2,
	\]
	\[
	 A_{i6}= - A_{(i+3)3}, \quad A_{6i}= - A_{3(i+3)}, \quad i=1,2.
	\]
	By Lemma \ref{lem:symx1}, we have that
	\[
		A_{ij}=A_{i+3,j+3}, \quad 2\leq i,j\leq 3, \quad A_{11}=A_{44},
	\]
	\[
		A_{ij}=A_{(i+3)(j-3)}, \quad i=2,3,\; j=5,6, \quad A_{14}=A_{41},
	\]
	\[
	 A_{1i}= - A_{4(i+3)}, \quad A_{i1}= - A_{(i+3)4}, \quad i=2,3,
	\]
	\[
	A_{1i}= - A_{4(i-3)}, \quad A_{i1}= - A_{(i-3)4}, \quad i=5,6.
	\]
	From the above identities, we conclude \eqref{eq2.04}.
	The other identities directly follow from the values of $\bvartheta_i$ and the expressions of the parameters defined in \eqref{def:alpha}.
\end{proof}

Next we analyze the behavior $A_{ij}$  as the distance $\varepsilon$ between $D_1$ and $D_2$ tends to zero, $i,j=1,\dots,6$.

\begin{prop}\label{thm-Aij}
The components of $A_{ij}$ given in Proposition \ref{prop:A} have the following expressions for $\varepsilon\ll 1$ and $|\omega|\ll 1$:
\begin{equation}\label{diag-2}
\begin{split}
\balpha_{11,1}=-\frac{\mu\pi}{\sqrt{\kappa\varepsilon}}+\varepsilon^{\frac{\alpha-1}{2}}\Ocal(1),\quad \balpha_{12,1}=\frac{\mu\pi}{\sqrt{\kappa\varepsilon}}+\varepsilon^{\frac{\alpha-1}{2}}\Ocal(1),\\
\balpha_{21,2}=-\frac{(\lambda+2\mu)\pi}{\sqrt{\kappa\varepsilon}}+\varepsilon^{\frac{\alpha-1}{2}}\Ocal(1),\quad \balpha_{22,2}=\frac{(\lambda+2\mu)\pi}{\sqrt{\kappa\varepsilon}}+\varepsilon^{\frac{\alpha-1}{2}}\Ocal(1);
\end{split}
\end{equation}
\begin{equation}\label{est-3132}
\beta_{3k}=\Ocal(1),\quad k=1,2,\quad \beta_{31}+\beta_{32}<0,\quad \Re(\beta_{31}-\beta_{32})<0,\quad \Im(\beta_{31}-\beta_{32})<0;
\end{equation}
and
\begin{equation}\label{est-311}
\balpha_{31,1}=\beta_{11}=\Ocal(1),\quad \balpha_{32,1}=-\beta_{12}=\Ocal(1),
\end{equation}
where $\kappa$ is the curvature of $\partial D$ at $(0,\varepsilon/2)$ and $(0,-\varepsilon/2)$, and $\alpha\in(0,1)$.
\end{prop}

The proof of Proposition \ref{thm-Aij} will be given in Subsection \ref{sec-grad}.
For further analysis, we provide the detailed estimates of the parameters $\balpha_1$, $\balpha_2$, and $\balpha_3$.
\begin{lem}\label{lem:esal}
	Let $C_{\Tcal}$ denote the norm of the operator $\Tcal$ defined in Lemma \ref{lem:ings}, and 
\[
	s_3 = \frac{\lambda + \mu}{4\mu(\lambda + 2\mu)\pi}, \quad
 s_4 = -\frac{ (\lambda+\mu)(1-16c_1\pi) + 2(\lambda+2\mu)(\ln{c_s} - 4\pi\tilde{\eta})  -2\mu\ln{c_p}}{8\mu(\lambda+2\mu)\pi }, 
\]
	with $c_1$ and $\tilde{\eta}$ given in \eqref{def:eta} and \eqref{def:tildeeta}, respectively. 
    If $\omega\ll 1$ is chosen such that $\frac{C_{\Tcal} + \abs{s_4}}{\abs{s_3}|\ln\omega|} <1$,
	then the parameters $\balpha_i$, $1\leq i\leq 3$, defined in \eqref{def:alpha}, have the following estimates:
	\begin{equation*}
		\balpha_i \leq \Ocal\left(\frac{1}{\ln\omega}\right) \quad \mbox{for} \quad i=1,2,3.
	\end{equation*}
	Moreover, if the domain $D$ is symmetric with respect to the origin, the parameters have the following detailed estimate

	\begin{equation}\label{eq:alpha-est}
		{\balpha_1}= {\balpha_2}= \frac{1}{2s_3\ln\omega}
		\begin{pmatrix}
			1\\
			1
		\end{pmatrix}(1+o(1)).
	\end{equation}
\end{lem}
\begin{proof}
	From the definition of $\bzeta_i$ in \eqref{def:xi} and the expression of the operator $\hat{\bS}_D^\omega$ in Lemma \ref{lem:assgl}, we have that 
	\begin{equation}\label{eq:pral1}
		\hat{\bS}_D^\omega[\bzeta_1] = (s_3 \ln\omega+s_4) \int_{\partial D} \bzeta_1+ 
        \bS_D[\bzeta_1] = \bvartheta_1.
	\end{equation}
	Next, we show that the function $\bzeta_1$ has the following asymptotic expansion
	\begin{equation}\label{eq:pral2}
		\bzeta_1(\by) = \bzeta_{1,0}(\by) + \sum_{j=1}^{\infty} \frac{1}{(\ln \omega)^j} \bzeta_{1,j}(\by).
	\end{equation}
	Substituting the last equation into \eqref{eq:pral1} and comparing the order of the parameter $\ln\omega$, we have that
	\begin{equation}\label{eq:zeta1e1}
		\int_{\partial D}\bzeta_{1,0} = 0,  \quad s_3\int_{\partial D}\bzeta_{1,1} + \bS_D[\bzeta_{1,0}] = \bvartheta_1,
	\end{equation}
	and for $j\geq 1$,
	\begin{equation}\label{eq:zeta1e2}
		\begin{cases}
			s_3\int_{\partial D}\bzeta_{1,j+1} + s_4 \int_{\partial D}\bzeta_{1,j} + \bS_D[\bzeta_{1,j}] = 0, \vspace{1mm} \\ 
			\int_{\partial D}\bzeta_{1,j} = \bt_j, 
		\end{cases}
	\end{equation}
	where $\bt_j$ are some vector-valued constants. 
	
	The existence of the solution $\bzeta_{1,0}$ in \eqref{eq:zeta1e1} is guaranteed by Lemma \ref{lem:ings}.
	Next, we consider the system \eqref{eq:zeta1e2}. By Lemma \ref{lem:ings}, the solution $\bzeta_{1,1}$ exists and satisfies the following estimate
	\[
		\|\bzeta_{1,1}\|_{L^2(\partial D)} + \abs{ s_3 \int_{\partial D}\bzeta_{1,2} + s_4 \bt_1 }  \leq C_{\Tcal} \abs{\bt_1},
	\]
	where $C_{\Tcal}$ is the norm of the operator $\Tcal$ defined in Lemma \ref{lem:ings}.
	From the last inequality and the triangle inequality, we have that
	\[
		\|\bzeta_{1,1}\|_{L^2(\partial D)} \leq C_{\Tcal}\abs{\bt_1} \quad \mbox{and} \quad \abs{\int_{\partial D}\bzeta_{1,2}} \leq \frac{C_{\Tcal} + \abs{s_4}}{\abs{s_3}} \abs{\bt_1}.
	\]
	By the same argument, we can show that
	\[
		\|\bzeta_{1,j}\|_{L^2(\partial D)} \leq C_{\Tcal} \left(  \frac{C_{\Tcal} + \abs{s_4}}{\abs{s_3}}\right)^{j-1}\abs{\bt_1} \quad \mbox{and} \quad \abs{\int_{\partial D}\bzeta_{1,j+1}} \leq \left(  \frac{C_{\Tcal} + \abs{s_4}}{\abs{s_3}}\right)^{j} \abs{\bt_1}.
	\]
	Thus, if $\frac{C_{\Tcal} + \abs{s_4}}{\abs{s_3}|\ln\omega|} <1$, we conclude that the asymptotic expansion \eqref{eq:pral2} converges with respect to the norm $\|\cdot\|_{L^2(\partial D)}$. Therefore, we have $\balpha_1 \leq \Ocal(1/\ln\omega)$ and the estimate of $\balpha_2$ and $\balpha_3$ follows from the same argument.
		
	Next, we give the estimate in \eqref{eq:alpha-est}. If the domain $D$ is symmetric with respect to the origin, by Lemma \ref{lem:symor} we have that similar to \eqref{eq:zeta1e1},
	\begin{equation}\label{eq:zeta1e3}
		\int_{\partial D}\bzeta_{4,0} = 0,  \quad s_3\bt_1 + \bS_D[\bzeta_{4,0}] = \bvartheta_4,
	\end{equation}
    where $\bt_1=\int_{\partial D}\bzeta_{1,1}$.
    Adding \eqref{eq:zeta1e3} and \eqref{eq:zeta1e1} gives that
	\begin{equation*}
		\int_{\partial D}\bzeta_{1,0} + \bzeta_{4,0} = 0,  \quad   \bS_D[\bzeta_{1,0} + \bzeta_{4,0}] + 2s_3 \bt_1 -
		\begin{pmatrix}
			1\\
			1
		\end{pmatrix}  = 0.
	\end{equation*}
    By Lemma \ref{lem:ings}, we conclude $\bt_1=\frac{1}{2s_3} (1,1)^t$ since the operator $\mathcal{T}$ is invertible.
	Finally, from the asymptotic expansion \eqref{eq:pral2}, we obtain the estimate of ${\balpha_1}$ in \eqref{eq:alpha-est}. Following the same argument, we can obtain the estimate of ${\balpha_2}$.
The proof is completed.
\end{proof}

Next, we define the parameters $\gamma_1$ and $\gamma_2$ by 
\begin{equation}\label{def:pga1}
	\gamma_1=\int_D\tilde{\bvartheta_1}(\by) \cdot \tilde{\bvartheta_3}(\by) \ \mathrm{d}\by = \int_{D_1}\by_2 \ \mathrm{d}\by, 
\end{equation}
and
\begin{equation}\label{def:pga2}
	\gamma_2=\int_D\tilde{\bvartheta_3}(\by) \cdot \tilde{\bvartheta_3}(\by) \ \mathrm{d}\by = \int_{D_1}\by_1^2 + \by_2^2  \ \mathrm{d}\by.
\end{equation}

\begin{lem}\label{lem:B}
	Define the matrix $B$ by 
	\[
	   B_{ij}=\int_D\tilde{\bvartheta_i}(\by)\cdot \tilde{\bvartheta_j}(\by) \ \mathrm{d}\by.
	\]
	If the domain $D$ is symmetric with respect to the origin and the $\bx_1$-axis, then the entries of the matrix $B$ have the following properties:
	\begin{equation}
	    \label{eq2.39}
		B_{ij}=0 \quad  \mbox{for}\, |i-j|\ge 3,
    \end{equation}
	\[
	   B_{12}=B_{21}=B_{45}=B_{54}=B_{23}=B_{32}=B_{56}=B_{65}=0,
	\]
	\begin{equation}
	    \label{eq2.41}
	   B_{13}=B_{31}=-B_{46}=-B_{64}=\gamma_1, \; B_{33}=B_{66}=\gamma_2, \; B_{ii}=\abs{D_1} \;  \mbox{for} \; i=1,2,4,5,
    \end{equation}
	where $\gamma_1$ and $\gamma_2$ are defined in \eqref{def:pga1} and \eqref{def:pga2}, respectively.
\end{lem}
\begin{proof}
	From the expression of the functions ${\bvartheta}_i$ and the definition of the entry $B_{ij}$, we have \eqref{eq2.39}
	and
	\[
	   B_{12}=B_{21}=B_{45}=B_{54}=0.
	\]
	Since the domain $D$ is symmetric with respect to the origin and the $\bx_1$-axis, we have that
	\[
		B_{23}=B_{32}=B_{56}=B_{65}=0.
	\]
 Finally, the direct calculation shows \eqref{eq2.41}.
\end{proof}

After these preparations, we are ready to investigate the resonant phenomena of the system \eqref{eq:or2}.
From Lemma \ref{lem:assgl}, the operator $\Acal(\omega, \delta )$ given in \eqref{eq:or} has the following expression 
\begin{equation}\label{def-Acal0}
\Acal(\omega,\delta)=\Acal_0+\Ocal(\omega^2\ln{\omega}+\delta ),
\end{equation}
where 
\begin{equation*}
	\Acal_0=  \left(
\begin{array}{cc}
	\hat{\bS}_{ D}^{\tau\omega} &  -\hat{\bS}_{ D}^{\omega}\medskip \\
	-\frac\bI{2} +  {\bK}_{ D}^{ *} & 0 \\
\end{array}
\right).
\end{equation*}
It is noted that the operator $\Acal_0$ has a nontrivial kernel. Indeed, from Lemma \ref{lem:ks0}, we find that the functions $(\bxi_i, \bzeta_i)^t$, $1\leq i\leq 6$, are in the kernel of $\Acal_0$.
Thus, by Gohberg-Sigal theory \cite{ak162}, for each $\delta$, there exists $\omega \in \mathbb{C}$ such that $\Acal(\omega,\delta)$ has a nontrivial kernel. The explicit expressions of the resonant frequencies and the corresponding eigenfunctions are given in the following theorem.

\begin{thm}\label{thm:res}
Assume that the domain $D$ is symmetric with respect to the origin and the $\bx_1$-axis. There exist six subwavelength resonant frequencies for the system \eqref{eq:or2}. 
The first resonant frequency satisfies
\[
  \omega_1 = \frac{1}{\tau} \sqrt{\frac{\delta (\balpha_{22,2}-\balpha_{21,2})}{\rho |D_1|}}(1 + o(1)),
\]
and the corresponding eigenfunctions are given by 
\[
	\bvarphi_1 =\bxi_2 -\bxi_5+\Ocal((\tau \omega_1)^2\ln{(\tau\omega_1)}), \quad \bpsi_1 =\bzeta_2 -\bzeta_5+\Ocal( \omega_1^2\ln\omega_1).
\]
The parameters $\balpha_i$ and $\balpha_{ij}$, $1\leq i\leq 6$ and $j=1,2$, are defined in \eqref{def:alpha}. The functions $\bxi_i$ and $\bzeta_i,1\leq i\leq 6,$ are given in \eqref{def:xi}.
The second resonant frequency is
\[
 \omega_{2}= \omega_{2,1} (1 + o(1)),
\]
where the leading order $\omega_{2,1}$ satisfies
\[
  2 s_3 \rho |D_1|  (\tau \omega_{2,1})^2 \ln{ \omega_{2,1}}  + \delta  = 0,
\]
with $s_3$ given in Lemma \ref{lem:esal}.
The corresponding eigenfunctions of $\omega_2$ are given by 
\[
	\bvarphi_2 =\bxi_2 +\bxi_5+\Ocal((\tau \omega_{2,1})^2\ln{(\tau\omega_{2,1})}), \quad \bpsi_2 =\bzeta_2 +\bzeta_5+\Ocal( \omega_{2,1}^2\ln\omega_{2,1}).
\]
The third resonant frequency is 
\[
  \omega_3 = \frac{1}{\tau} \sqrt{\frac{ (\beta_{32} - \beta_{31}) |D_1|}{\left( |D_1| \gamma_{2} - \gamma_{1}^{2} \right)\rho } \delta }(1 + o(1)).
\]
The corresponding eigenfunctions are given by
\begin{align*}
	\bvarphi_3 &= e_{3,1}\bxi_1 + e_{3,3}\bxi_3 + e_{3,4}\bxi_4 + e_{3,6}\bxi_6 + \Ocal((\tau \omega_3)^2\ln{(\tau\omega_3)}), \\
	\bpsi_3 &= e_{3,1}\bzeta_1 + e_{3,3}\bzeta_3 + e_{3,4}\bzeta_4 + e_{3,6}\bzeta_6 + \Ocal( \omega_3^2\ln\omega_3),
\end{align*}
with $e_{3,i}$ given in \eqref{def-e3}, $i=1,3,4,6$. The fourth resonant frequency is 
\[
 \omega_{4}= \omega_{4,1} (1 + o(1)),
\]
where the leading order $\omega_{4,1}$ satisfies
\[
  2 s_3 \rho |D_1|  (\tau \omega_{4,1})^2 \ln{ \omega_{4,1}}  + \delta  = 0.
\]
The corresponding eigenfunctions are given by
\begin{align*}
	\bvarphi_4 &= e_{4,1}\bxi_1 + e_{4,3}\bxi_3 + e_{4,4}\bxi_4 + e_{4,6}\bxi_6 + \Ocal((\tau \omega_4)^2\ln{(\tau\omega_4)}), \\
	\bpsi_4 &= e_{4,1}\bzeta_1 + e_{4,3}\bzeta_3 + e_{4,4}\bzeta_4 + e_{4,6}\bzeta_6 + \Ocal( \omega_4^2\ln\omega_4),
\end{align*}
where $e_{4,i}$ are defined in \eqref{def-e4}, $i=1,3,4,6$. The fifth resonant frequency is
\[
	\omega_5 = \frac{1}{\tau} \sqrt{\frac{ (\balpha_{12,1} - \balpha_{11,1}) \gamma_2 }{(|D_{1}|\, \gamma_{2} - \gamma_{1}^2)\rho } \delta }(1 + o(1)).
\]
The corresponding eigenfunctions are given by
\begin{align*}
	\bvarphi_5 &= e_{5,1}\bxi_1 + e_{5,3}\bxi_3 + e_{5,4}\bxi_4 + e_{5,6}\bxi_6 + \Ocal((\tau \omega_5)^2\ln{(\tau\omega_5)}), \\
	\bpsi_5 &= e_{5,1}\bzeta_1 + e_{5,3}\bzeta_3 + e_{5,4}\bzeta_4 + e_{5,6}\bzeta_6 + \Ocal( \omega_5^2\ln\omega_5),
\end{align*}
where $e_{5,i}$ are defined in \eqref{def-e5}, $i=1,3,4,6$. The sixth resonant frequency is
\[
	\omega_6 = \frac{1}{\tau} \sqrt{\frac{\beta_{31} + \beta_{32}}{-\gamma_{2} \rho} \delta }(1 + o(1)).
\]
The corresponding eigenfunctions are given by
\begin{align*}
	\bvarphi_6 &= e_{6,1}\bxi_1 + e_{6,3}\bxi_3 + e_{6,4}\bxi_4 + e_{6,6}\bxi_6 + \Ocal((\tau \omega_6)^2\ln{(\tau\omega_6)}), \\
	\bpsi_6 &= e_{6,1}\bzeta_1 + e_{6,3}\bzeta_3 + e_{6,4}\bzeta_4 + e_{6,6}\bzeta_6 + \Ocal( \omega_6^2\ln\omega_6),
\end{align*}
where $e_{6,i}$ are defined in \eqref{def-e6}, $i=1,3,4,6$.
\end{thm}

\begin{proof}
From the discussion above the theorem, we conclude that the system \eqref{eq:or2} shares resonant frequencies; that is, there exist resonant frequencies $\omega\in\mathbb{C}$ such that $\Acal(\omega,\delta)$ defined in \eqref{eq:or} has a nontrivial kernel. Assume that $(\bvarphi, \bpsi)^t$ is in the kernel, namely,  
\[
\Acal(\omega,\delta)\begin{bmatrix}
	\bvarphi \\  \bpsi 
\end{bmatrix}=0. 
\]
From the expression of the operator $\Acal(\omega,\delta)$, the last equation is equivalent to
\begin{equation*}
\begin{cases}
	\bS_D^{\tau \omega}[\bvarphi]-\bS_D^\omega [\bpsi]=0, \vspace{1mm} \\
	\left(-\frac{1}{2}{\bf I}+\bK_D^{\tau \omega,*}\right)[\bvarphi]-\delta\left( \frac{1}{2}{\bf I}+\bK^{\omega,*}_D\right)[\bpsi]=0.
\end{cases}
\end{equation*}
By the asymptotic expansion of the operators $\bS_D^{\omega}$ and $\bK_D^{\tau \omega,*}$ in Lemma \ref{lem:assgl}, we have that
\begin{equation}\label{eq:eee1}
	\hat{\bS}_D^{\tau \omega} [\bvarphi]-\hat{\bS}^\omega [\bpsi] =\Ocal(\omega^2\ln{\omega}),
\end{equation}
and 
\begin{equation}\label{eq:eee3}
	\begin{split}
	&\left(-\frac{1}{2}{\bf I} +\bK_D^*+(\tau \omega)^2\ln{(\tau \omega)}\bK_{D,1}^{(1)} + (\tau \omega)^2 \bK_{D,1}^{(2)}\right)[\bvarphi]-\delta \left( \frac{1}{2}{\bf I}+\bK_D^*\right)[\bpsi] \\
	= & \Ocal\left((\tau \omega)^4 \ln{(\tau \omega)}+\delta \omega^2\ln{\omega}\right).
	\end{split}
\end{equation}
From \eqref{eq:eee3}, we have that
\begin{equation*}
	\bvarphi -\Ocal((\tau \omega)^2\ln{(\tau\omega)}+\delta) \in \ker{\left(-\frac{1}{2}{\bf I} +\bK_D^*\right)}.
\end{equation*}
By Lemma \ref{lem:ks0}, the function $\bvarphi$ can be written as 
\begin{equation*}
	\bvarphi =\sum_{j=1}^6 d_j\bxi_j+\Ocal((\tau \omega)^2\ln{(\tau\omega)}+\delta),
\end{equation*}
where $d_j$, $1\leq j\leq 6$, are some constants needed to be determined.
Then, from \eqref{def:xi} and \eqref{eq:eee1}, the function $\bpsi$ can be expressed by 
\begin{equation*}
	\bpsi =\sum_{j=1}^6d_j\bzeta_j +\Ocal( \omega^2\ln\omega+\delta).  
\end{equation*}
Substituting the last two equations into \eqref{eq:eee3} gives that 
\begin{equation}\label{eq:gere1}
	\begin{split}
		& \left(-\frac{1}{2}{\bf I} +\bK_D^*\right)[\bvarphi]+ (\tau\omega)^2\ln{(\tau\omega)}\bK_{D,1}^{(1)}\left(\sum_{j=1}^6 d_j\bxi_j \right)  + (\tau\omega)^2 \bK_{D,1}^{(2)}\left(\sum_{j=1}^6 d_j\bxi_j \right) \\
	    &-\delta \left(\frac{1}{2}{\bf I}+\bK_D^*\right)\left( \sum_{j=1}^6 d_j\bzeta_j \right) =\Ocal\left( (\tau \omega)^4 (\ln{(\tau \omega)})^2 +\delta \omega^2 \ln{\omega} +\delta^2 \right).
    \end{split}
\end{equation}
Multiplying ${\bm{\vartheta}}_i$, $1\leq i\leq 6$, on both sides of \eqref{eq:gere1}, integrating on $\partial D$ and with the help of Lemma \ref{lem:k0} gives that 
\[
    \begin{split}
	&(\tau \omega)^2\ln{(\tau\omega)} \int_{\partial D} \bK_{D,1}^{(1)}\left(\sum_{j=1}^6 d_j\bxi_j\right) \cdot {\bm{\bvartheta}}_i\  + (\tau \omega)^2 \int_{\partial D} \bK_{D,1}^{(2)}\left(\sum_{j=1}^6 d_j\bxi_j\right) \cdot {\bm{\bvartheta}}_i\  \\
	& -\delta \int_{\partial D} \left(\frac{1}{2}{\bf I}+\bK_D^*\right)\left( \sum_{j=1}^6 d_j\bzeta_j \right) \cdot{\bm{\vartheta}}_i\  =\Ocal\left( (\tau \omega)^4 (\ln{(\tau \omega)})^2 +\delta \omega^2 \ln{\omega} +\delta^2 \right).
	\end{split}
\]
By Lemmas \ref{lem:k0}, \ref{lem:ks01} and \ref{lem:ks02}, the last equation can be written as, for $1\leq i\leq 6$, 
\begin{equation*}
	\begin{split}
	& \sum_{j=1}^6 d_j\left( \rho (\tau \omega)^2 \int_D\tilde{\bvartheta_i}(\by) \cdot \tilde{\bvartheta_j}(\by)    + \delta \int_{\partial D} {\bm{\bvartheta}}_i(\by)  \cdot  \bzeta_j(\by)  \right)\\
	&= \Ocal\left( (\tau \omega)^4 (\ln{(\tau \omega)})^2 +\delta \omega^2 \ln{\omega} +\delta^2 \right).
	\end{split}
\end{equation*}
Using the notations given in Proposition \ref{prop:A} and Lemma \ref{lem:B}, namely, $A_{ij}=\int_{\partial D} {\bm{\bvartheta}}_i(\by) \cdot \bzeta_j(\by) \ \mathrm{d}s(\by) $ and $B_{ij}=\int_D\tilde{\bvartheta_i}(\by) \cdot \tilde{\bvartheta_j}(\by) \ \mathrm{d}\by$, the leading terms of the last equations can be written as
\begin{equation}\label{eq:adb1}
	\left(\rho (\tau \omega)^2 B + \delta A \right)\bd = 0,
\end{equation}
where $\bd = (d_1, d_2, d_3, d_4, d_5, d_6)^t$. Then, by Lemma \ref{lem:symor}, Proposition \ref{prop:A}, and Lemma \ref{lem:B}, \eqref{eq:adb1} can be decomposed into two sub-equations
\begin{equation}\label{eq:adb2}
	\left(\rho (\tau \omega)^2 B_i + \delta A_i \right)\tilde{\bd}_i = 0, \quad i=1,2,
\end{equation}
where $A_1$ and $B_1$ are 2-by-2 matrices, and $A_2$ and $B_2$ are 4-by-4 matrices. 

For matrices $A_1$ and $B_1$, they are given by
\[
	B_1=\begin{pmatrix}
		|D_1|& 0\\
		0& |D_1|
	\end{pmatrix}, \quad  \quad 
	A_1=\begin{pmatrix}
		\balpha_{21,2} & \balpha_{22,2}\\
		\balpha_{22,2} & \balpha_{21,2}
	\end{pmatrix}.
\]
Direct calculation shows that the eigensystem of the matrix $A_1$ is given by
\[
	\balpha_{21,2}-\balpha_{22,2}, \; (1, -1)^t; \quad \balpha_{21,2}+\balpha_{22,2}, \; (1, 1)^t.
\]
Thus, from equation \eqref{eq:adb2}, the first resonant frequency satisfies
\[
  \omega_1 = \frac{1}{\tau} \sqrt{\frac{\delta (\balpha_{22,2}-\balpha_{21,2})}{\rho |D_1|}}(1 + o(1)),
\]
The corresponding eigenfunctions are given by 
\[
	\bvarphi_1 =\bxi_2 -\bxi_5+\Ocal((\tau \omega_1)^2\ln{(\tau\omega_1)}), \quad \bpsi_1 =\bzeta_2 -\bzeta_5+\Ocal( \omega_1^2\ln\omega_1),
\]
where we used the fact $\delta = \Ocal((\tau \omega_1)^2)$.
From Lemma \ref{lem:esal}, the leading order of the second resonant frequency satisfies
\[
	2 s_3 \rho |D_1|  (\tau \omega_{2,1})^2 \ln{ \omega_{2,1}}  + \delta  = 0.
\]
The corresponding eigenfunctions are given by 
\[
	\bvarphi_2 =\bxi_2 +\bxi_5+\Ocal((\tau \omega_2)^2\ln{(\tau\omega_2)}), \quad \bpsi_2 =\bzeta_2 +\bzeta_5+\Ocal( \omega_2^2\ln\omega_2),
\]
which follows from the fact $\delta = \Ocal((\tau \omega_2)^2\ln \omega_2)$.

For matrices $A_2$ and $B_2$, from Lemma \ref{lem:symor}, they have the following expressions:
\[
	B_2 =\begin{pmatrix}
		|D_1| & \gamma_1 & 0& 0\\
		\gamma_1 & \gamma_2 & 0 & 0\\
		0&0&|D_1|&-\gamma_1&\\
		0&0& -\gamma_1 &\gamma_2
	\end{pmatrix}, \quad 
	A_2=\begin{pmatrix}
		\balpha_{11,1}  &\balpha_{31,1} &\balpha_{12,1}  &-\balpha_{32,1}  \\  
		\balpha_{31,1}  &\beta_{31} &\balpha_{32,1}  &\beta_{32}\\
		\balpha_{12,1}  &\balpha_{32,1} &\balpha_{11,1}  &-\balpha_{31,1}  \\  
		-\balpha_{32,1}  &\beta_{32} &-\balpha_{31,1} &\beta_{31}
	\end{pmatrix}.
\]
To find the resonant frequencies, we consider the following equation instead of \eqref{eq:adb2}:
\begin{equation}\label{eq:adb3}
	\left(\rho (\tau \omega)^2  + \delta B_2^{-1} A_2 \right)\tilde{\bd}_2 = 0.
\end{equation}
The first eigenvalue of the matrix $B_2^{-1} A_2$ is given by
\begin{equation*}
		\iota_1 =  \frac{\left((\beta_{32} - \beta_{31}) |D_1|
		- (\balpha_{11,1} + \balpha_{12,1}) \gamma_{2}
		+ 2 (\balpha_{31,1} + \balpha_{32,1} ) \gamma_{1}
		  + \sqrt{\iota_{1,1}} \right)}{-2\left( |D_1| \gamma_{2} - \gamma_{1}^{2} \right) } ,
	\end{equation*}
where 
\begin{equation}\label{eq:iota11}
	\begin{split}
		\iota_{1,1} = &((\beta_{31} - \beta_{32}) |D_1| + (\balpha_{11,1} + \balpha_{12,1}) \gamma_{2} - 2(\balpha_{31,1} + \balpha_{32,1} ) \gamma_{1})^2
	+ \\
	& 4 ((\balpha_{31,1} + \balpha_{32,1})^2 - (\balpha_{11,1} + \balpha_{12,1})(\beta_{31} - \beta_{32})) (|D_1|\gamma_{2} - \gamma_{1}^2).
	\end{split}
\end{equation}
By Lemma \ref{lem:esal}, the eigenvalue $\iota_1$ has the following asymptotic expansion
\begin{equation*}
	\iota_1 = \frac{ (\beta_{32} - \beta_{31}) |D_1|}{-\left( |D_1| \gamma_{2} - \gamma_{1}^{2} \right) } \big(1 + o(1) \big).
\end{equation*}
Thus, from \eqref{eq:adb3}, the third resonant frequency satisfies
\[
  \omega_3 = \frac{1}{\tau} \sqrt{\frac{ (\beta_{32} - \beta_{31}) |D_1|}{\left( |D_1| \gamma_{2} - \gamma_{1}^{2} \right)\rho } \delta }(1 + o(1)).
\]
The corresponding eigenfunctions are given by
\begin{align*}
	\bvarphi_3 &= e_{3,1}\bxi_1 + e_{3,3}\bxi_3 + e_{3,4}\bxi_4 + e_{3,6}\bxi_6 + \Ocal((\tau \omega_3)^2\ln{(\tau\omega_3)}), \\
	\bpsi_3 &= e_{3,1}\bzeta_1 + e_{3,3}\bzeta_3 + e_{3,4}\bzeta_4 + e_{3,6}\bzeta_6 + \Ocal( \omega_3^2\ln\omega_3),
\end{align*}
where 
\begin{equation}\label{def-e3}
e_{3,1} = e_{3,4} =\frac{\gamma_1}{|D_1|}(1 + o(1)), \quad e_{3,3}=-1 +o(1), \quad e_{3,6}=1 +o(1). 
\end{equation}

The second eigenvalue of the matrix $B_2^{-1} A_2$ is given by
\begin{equation*}
	\iota_2 =  \frac{\left((\beta_{32} - \beta_{31}) |D_1|
		- (\balpha_{11,1} + \balpha_{12,1}) \gamma_{2}
		+ 2 (\balpha_{31,1} + \balpha_{32,1} ) \gamma_{1}
		  - \sqrt{\iota_{1,1}} \right)}{-2\left( |D_1| \gamma_{2} - \gamma_{1}^{2} \right) } ,
\end{equation*}
where $\iota_{1,1}$ is defined in \eqref{eq:iota11}.
By Lemma \ref{lem:esal} and direct calculation, the eigenvalue $\iota_2$ has the following asymptotic expansion
\begin{equation*}
	\iota_2 = \frac{\balpha_{11,1} + \balpha_{12,1}}{|D_1|} \big(1 + o(1) \big).
\end{equation*}
Thus, from \eqref{eq:adb3} and Lemma \ref{lem:esal}, the leading order of the fourth resonant frequency satisfies
\[
	2 s_3 \rho |D_1|  (\tau \omega_{4,1})^2 \ln{ \omega_{4,1}}  + \delta  = 0.
\]
The corresponding eigenfunctions are given by
\begin{align*}
	\bvarphi_4 &= e_{4,1}\bxi_1 + e_{4,3}\bxi_3 + e_{4,4}\bxi_4 + e_{4,6}\bxi_6 + \Ocal((\tau \omega_4)^2\ln{(\tau\omega_4)}), \\
	\bpsi_4 &= e_{4,1}\bzeta_1 + e_{4,3}\bzeta_3 + e_{4,4}\bzeta_4 + e_{4,6}\bzeta_6 + \Ocal( \omega_4^2\ln\omega_4),
\end{align*}
where 
\begin{equation}\label{def-e4}
  e_{4,1} = e_{4,4} =\frac{\beta_{32} - \beta_{31}}{\balpha_{31,1} - \balpha_{32,1}}(1 + o(1)), \quad e_{4,3}=-1 +o(1), \quad e_{4,6}=1 +o(1). 
\end{equation}

The third eigenvalue of the matrix $B_2^{-1} A_2$ is given by
\begin{equation*}
		\iota_3 =  \frac{(-\beta_{32} - \beta_{31}) |D_1|
		+ (\balpha_{12,1} - \balpha_{11,1}) \gamma_{2}
		+ 2(\balpha_{31,1} - \balpha_{32,1}) \gamma_{1}
		  + \sqrt{\iota_{3,1}} }{-2\left( |D_1| \gamma_{2} - \gamma_{1}^{2} \right) } ,
	\end{equation*}
where 
\begin{equation}\label{eq:iota31}
	\begin{split}
		\iota_{3,1} = &((\beta_{31} + \beta_{32}) |D_1| + (\balpha_{11,1} - \balpha_{12,1}) \gamma_{2} - 2(\balpha_{31,1} - \balpha_{32,1} ) \gamma_{1})^2
	+ \\
	& 4 ((\balpha_{31,1} - \balpha_{32,1})^2 - (\balpha_{11,1} - \balpha_{12,1})(\beta_{31} + \beta_{32})) (|D_1|\gamma_{2} - \gamma_{1}^2).
	\end{split}
\end{equation}
By Proposition \ref{thm-Aij}, the eigenvalue $\iota_3$ has the following asymptotic expansion
\begin{equation*}
	\iota_3 = \frac{ (\balpha_{12,1} - \balpha_{11,1}) \gamma_2 }{-(|D_{1}|\, \gamma_{2} - \gamma_{1}^2) } (1 + o(1)),
\end{equation*}
Thus, from \eqref{eq:adb3}, the fifth resonant frequency satisfies
\[
	\omega_5 = \frac{1}{\tau} \sqrt{\frac{ (\balpha_{12,1} - \balpha_{11,1}) \gamma_2 }{(|D_{1}|\, \gamma_{2} - \gamma_{1}^2)\rho } \delta }(1 + o(1)).
\]
The corresponding eigenfunctions are given by
\begin{align*}
	\bvarphi_5 &= e_{5,1}\bxi_1 + e_{5,3}\bxi_3 + e_{5,4}\bxi_4 + e_{5,6}\bxi_6 + \Ocal((\tau \omega_5)^2\ln{(\tau\omega_5)}), \\
	\bpsi_5 &= e_{5,1}\bzeta_1 + e_{5,3}\bzeta_3 + e_{5,4}\bzeta_4 + e_{5,6}\bzeta_6 + \Ocal( \omega_5^2\ln\omega_5),
\end{align*}
where 
\begin{equation}\label{def-e5}
  e_{5,1} = -e_{5,4} = 1 +o(1), \quad e_{5,3}= e_{5,6}= \frac{(\balpha_{11,1} - \balpha_{12,1})\gamma_{2}}{ -(\balpha_{31,1} - \balpha_{32,1})\gamma_{2}  + (\beta_{31} + \beta_{32}) \gamma_{1}  }(1 + o(1)). 
\end{equation}

The fourth eigenvalue of the matrix $B_2^{-1} A_2$ is given by
\begin{equation*}
	\iota_4 =  \frac{(-\beta_{32} - \beta_{31}) |D_1|
	+ (\balpha_{12,1} - \balpha_{11,1}) \gamma_{2}
	+ 2(\balpha_{31,1} - \balpha_{32,1}) \gamma_{1}
	  - \sqrt{\iota_{3,1}} }{-2\left( |D_1| \gamma_{2} - \gamma_{1}^{2} \right) } ,
	\end{equation*}
where $\iota_{3,1}$ is defined in \eqref{eq:iota31}.
By Proposition \ref{thm-Aij}, the eigenvalue $\iota_4$ has the following asymptotic expansion
\begin{equation*}
	\iota_4 = \frac{\beta_{31} + \beta_{32}}{\gamma_{2}} (1 + o(1)).
\end{equation*}
Thus, from \eqref{eq:adb3}, the sixth resonant frequency satisfies
\[
	\omega_6 = \frac{1}{\tau} \sqrt{\frac{\beta_{31} + \beta_{32}}{-\gamma_{2} \rho} \delta }(1 + o(1)).
\]
The corresponding eigenfunctions are given by
\begin{align*}
	\bvarphi_6 &= e_{6,1}\bxi_1 + e_{6,3}\bxi_3 + e_{6,4}\bxi_4 + e_{6,6}\bxi_6 + \Ocal((\tau \omega_6)^2\ln{(\tau\omega_6)}), \\
	\bpsi_6 &= e_{6,1}\bzeta_1 + e_{6,3}\bzeta_3 + e_{6,4}\bzeta_4 + e_{6,6}\bzeta_6 + \Ocal( \omega_6^2\ln\omega_6),
\end{align*}
where 
\begin{equation}\label{def-e6}
  e_{6,1} = -e_{6,4} =-\frac{ \gamma_2 }{ \gamma_1 }(1 + o(1)), \quad e_{6,3}= e_{6,6}=1 +o(1). 
\end{equation}
The proof is completed. 
\end{proof}

\begin{rem}\label{rem:res}
We provide some explanation and estimates of the resonant frequencies given in Theorem \ref{thm:res}. Some resonant frequencies contain the term $|D_{1}|\, \gamma_{2} - \gamma_{1}^2$ and we can show that the term is positive. Indeed, by using \eqref{def:pga1}, \eqref{def:pga2}, and H\"{o}lder's inequality, we have 
\begin{align*}
\gamma_{1}^{2}=\Big(\int_{D_1}\by_2 \ \mathrm{d}\by\Big)^2\leq |D_1|\int_{D_1}\by_2^2 \ \mathrm{d}\by\leq |D_1| \gamma_{2}.
\end{align*}
The equality holds if and only if $\by_1=0$ and $\by_2$ is a constant. Thus, we have $|D_{1}|\, \gamma_{2} - \gamma_{1}^2>0$. 
Then from Proposition \ref{thm-Aij}, the first and the fifth resonant frequencies have the following estimates for $\varepsilon\ll 1$,
\[
  \omega_1 = \frac{1}{\tau} \sqrt{\frac{ 2\pi (\lambda + 2\mu)}{\rho |D_1|}} \frac{\sqrt{\delta}}{(\kappa\varepsilon)^{1/4}}(1 + o(1)), \quad 
  \omega_5 = \frac{1}{\tau} \sqrt{\frac{ 2\pi\mu \gamma_2 }{(|D_{1}|\, \gamma_{2} - \gamma_{1}^2)\rho } } \frac{\sqrt{\delta}}{(\kappa\varepsilon)^{1/4}}(1 + o(1)).
\]
Furthermore, the resonant frequencies $\omega_i$, $i=1,5$, are of the order $\Ocal(\delta^{(1-\beta/2)/2})$ if we choose the parameter $\varepsilon=\Ocal(\delta^{\beta})$ for some $0<\beta<2$.
For the resonant frequencies $\omega_i$, $i=2,4$, their leading order satisfies 
\[
2 s_3 \rho |D_1|  (\tau \omega)^2 \ln{ \omega}  + \delta  = 0.
\]
The the resonant frequencies $\omega_i$, $i=3,6$, are of the order $\Ocal(\delta^{1/2})$.
\end{rem}

\section{Blowup analysis of eigenmodes gradients}\label{sec:blowup}

This section is devoted to analyzing the gradient behavior of eigenmodes. To this end, we first establish several gradient estimates, which will be used both in the subsequent analysis and in proving Proposition \ref{thm-Aij} in Subsection \ref{sec-grad}. We then proceed to a systematic study of eigenmode gradients as the contrast parameter $\delta>0$ is small enough in Subsection \ref{sec-eigenmodes}.


\subsection{Gradient estimates and the proof of Proposition \ref{thm-Aij}}\label{sec-grad}

We first present more characteristics of the two domains $D_1$ and $D_2$. 
By a translation and rotation of coordinates (if necessary), there exists a constant $R_0$ independent of $\varepsilon$, such that the sections of $\partial D_{1}$ and $\partial D_{2}$ near the origin, respectively, can be represented by
\begin{align}\label{h1h2}
	{\bf x}_{2}=\frac{\varepsilon}{2}+\mathcal{H}_{1}({\bf x}_1)\quad\mbox{and}\quad {\bf x}_{2}=-\frac{\varepsilon}{2}+\mathcal{H}_{2}({\bf x}_1)\quad\mbox{for}~|{\bf x}_1|<2R_0,
\end{align}
where $\varepsilon:=\mbox{dist}(D_1,D_2)$. In \eqref{h1h2}, the functions $\mathcal{H}_i\in C^{2,\alpha}(-2R_0,2R_0)$, $i=1,2$, have the expressions 
\begin{align*}
	\mathcal{H}_{i}({\bf x}_1)=
	(-1)^{i+1}\frac{\kappa}{2}|{\bf x}_1|^{2}+\Ocal(|{\bf x}_1|^{2+\alpha}),
\end{align*}
and 
\begin{equation*}
	\|\mathcal{H}_{i}\|_{C^{2,\alpha}(-2R_0,2R_0)}\leq C,
\end{equation*}
where $C$ is a positive constant independent of $\varepsilon$ and $\kappa$ is the curvature of $\partial D$ at $(0,\varepsilon/2)$ and $(0,-\varepsilon/2)$. Throughout this section, the constant $C$ is independent of $\varepsilon$, and may vary from line to line in various inequalities. Here we would like to remark that our method can be applied to deal with the
more general inclusions case, say, $\mathcal{H}_{i}({\bf x}_1)=
(-1)^{i+1}\kappa_i|{\bf x}_1|^{2}+\Ocal(|{\bf x}_1|^{2+\alpha})$ with two positive constants $\kappa_1$ and $\kappa_2$.
For $0<r\leq\,2R_0$, we define the narrow region between $\partial{D}_{1}$ and $\partial{D}_{2}$ as follows:
\begin{equation}\label{narrowreg}
	\Omega_r:=\left\{({\bf x}_1,{\bf x}_2)\in \mathbb{R}^{2}: -\frac{\varepsilon}{2}+\mathcal{H}_{2}({\bf x}_1)<{\bf x}_2<\frac{\varepsilon}{2}+\mathcal{H}_{1}({\bf x}_1),~|{\bf x}_1|<r\right\},
\end{equation}
and the vertical distance between $\partial{D}_{1}$ and $\partial{D}_{2}$ is denoted by
\begin{equation}\label{delta_x'}
	\delta({\bf x}_1):=\varepsilon+\mathcal{H}_{1}({\bf x}_1)-\mathcal{H}_{2}({\bf x}_1)=\varepsilon+\kappa|{\bf x}_1|^{2}+\Ocal(|{\bf x}_1|^{2+\alpha}).
\end{equation}

Now we take a large constant $R$  such that $\overline{D_1\cup D_2}\subset B_R$. It follows from \eqref{def:xi} that  ${\bf w}_i:=\hat{\bS}_D^\omega[\bzeta_i]$ satisfies the Dirichlet problem as follows:
\begin{align}\label{eq-wi}
\begin{cases}
\mathcal{L}_{ {\lambda}, {\mu}}{\bf w}_i=0,  &  \bx\in D^{e}:=\mathbb R^2\setminus\overline{D},\\
{\bf w}_i={\bm{\vartheta}}_i,    & \bx\in\partial D,\\
{\bf w}_i(\bx)={\bf f}_i(\bx),& \bx\in\partial B_R,
\end{cases}
\end{align}
where ${\bf f}_i(\bx)=\hat{\bS}_D^\omega[\bzeta_i]$, $i=1,\dots,6$. By using 
\begin{equation*}
	\partial_{\bnu}   \hat{\bS}_{ D}^{\omega}[\bvarphi]|_{\pm}(\bx)=\left( \pm\frac{1}{2}\bI +  \bK_{ D}^{\omega, *} \right)[\bvarphi](\bx), \quad \bx\in\partial D,
\end{equation*}
and $(-\mathbf{I}/2 +  \bK^{\omega,*}_{D})[\bzeta_i]=0$, we have 
\begin{align*}
\int_{\partial D} \bzeta_i{\bm{\vartheta}}_j=\int_{\partial D}\left(\frac{1}{2}\bI +  \bK_{D}^{\omega,*}\right)[\bzeta_i]{\bm{\vartheta}}_j
=\int_{\partial D}\frac{\partial\hat{\bS}_{ D}^\omega[\bzeta_i]}{\partial{\bnu}}\Bigg|_{+}{\bm{\vartheta}}_j=
\int_{\partial D}\frac{\partial\bw_i}{\partial{\bnu}}\Bigg|_{+}{\bm{\vartheta}}_j.
\end{align*}
This implies that the key point in the analysis of $\int_{\partial D} \bzeta_i{\bm{\vartheta}}_j$ is to estimate the term $\frac{\partial\bw_i}{\partial{\bnu}}$. 

By using the gradient estimates for elliptic systems (see, for instance, \cite{ADN1959}), we have 
\begin{equation*}
\|\nabla \bw_i\|_{L^\infty(\mathbb R^2\setminus\overline B_R)}\leq C,\quad i=1,\dots,6,
\end{equation*}
where $C>0$ is a constant independently of $\varepsilon$ and $\bw_i$ is the solution of \eqref{eq-wi}. Thus, we shall consider the problem in $B_R\setminus\overline{D}=:\Omega$. We next decompose 
\begin{equation}\label{decom-wi}
\bw_i=\bw_{i,1}+\bw_{i,2},
\end{equation}
where $\bw_{i,1}$ and $\bw_{i,2}$, respectively, satisfy 
\begin{align}\label{eq-wi1}
\begin{cases}
\mathcal{L}_{ {\lambda}, {\mu}}{\bf w}_{i,1}=0,  &  \bx\in \Omega,\\
{\bf w}_{i,1}={\bm{\vartheta}}_i,    & \bx\in\partial D,\\
{\bf w}_{i,1}=0,& \bx\in\partial B_R,
\end{cases}
\end{align}
and 
\begin{align}\label{eq-wi2}
\begin{cases}
\mathcal{L}_{ {\lambda}, {\mu}}{\bf w}_{i,2}=0,  &  \bx\in \Omega,\\
{\bf w}_{i,2}=0,    & \bx\in\partial D,\\
{\bf w}_{i,2}={\bf f}_i,& \bx\in\partial B_R.
\end{cases}
\end{align}

We shall establish the asymptotics of $\nabla\bw_{i,1}$ as follows.

\begin{lem}\label{prop-wi1}
Let $\bw_{i,1}$ be the weak solutions of \eqref{eq-wi1}, $i=1,\dots,6$. Then for sufficiently small  $\varepsilon>0$ and $\bx\in\Omega_{R_0}$, we have
\begin{equation*}
\nabla\bw_{i,1}=\nabla\bv_i+\Ocal(1),
\end{equation*}
where $\bv_i$ are given in \eqref{auxiliary improved}.
\end{lem}

The key idea in proving Lemma \ref{prop-wi1} is to construct suitable auxiliary functions and then apply energy estimates together with an iterative technique, as in the proof of \cite[Theorems 3.2 and 3.3]{LX2025}. To this end, we introduce an auxiliary function $p\in C^{2,\alpha}(\mathbb{R}^2)$, 
\begin{equation}\label{def-p}
p(\bx)=\frac{x_{2}+\frac{\varepsilon}{2}-\mathcal{H}_{2}({\bf x}_1)}{\delta({\bf x}_1)}\quad\hbox{in}\ \Omega_{2R_0},
\end{equation}
such that $p=1$ on $\partial{D}_{1}$, $p=0$ on $\partial{D}_{2}\cup\partial B_R$, and $\|p\|_{C^{2,\alpha}(D^e\setminus\Omega_{R_0})}\leq\,C$, where $\Omega_{R_0}$ and $\delta({\bf x}_1)$ are defined in \eqref{narrowreg} and \eqref{delta_x'}, respectively. Next, we define $\bv_{i}\in C^{2,\alpha}(D^e)$ such that $\bv_{i}=\bw_{i}$ on $\partial D$, $\bv_{i}=0$ on $\partial B_R$, $\|\bv_{i}\|_{C^{2,\alpha}(D^e\setminus\Omega_{R_0})}\leq C$, and for $\bx\in\Omega_{2R_0}$,
\begin{equation}\label{auxiliary improved}
\begin{split}
\bv_{1}(\bx)&:=p(\bx)\bvarkappa_1
+\frac{\lambda+\mu}{\lambda+2\mu}f(p(\bx))\, \delta'({\bf x}_1)\,\bvarkappa_2,\\
\bv_{2}(\bx)&:=p(\bx)\bvarkappa_2
+\frac{\lambda+\mu}{\mu}f(p(\bx))\, \delta'({\bf x}_1)\,\bvarkappa_1,\\
\bv_{3}(\bx)&:=p(\bx)\bvarkappa_3,
\end{split}
\end{equation} 
where $\bvarkappa_i$ are defined in \eqref{eq:dpsi}, \begin{equation*}
f(p):=\dfrac{1}{2}\Big(p-\frac{1}{2}\Big)^{2}-\dfrac{1}{8}.
\end{equation*} 
Similarly, we define $\tilde p=1-p$. Then we take the auxiliary functions $\bv_{i}(\bx)$ by replacing $p$ with $\tilde p$ in \eqref{auxiliary improved}, $i=4,5,6$.  The correction terms involving $f$ allow us to capture all singular contributions in $\nabla\bw_{i,1}$, ensuring that the remaining terms $\nabla(\bw_{i,1}-\bv_i)$ are of order $\Ocal(1)$. For further details, we refer the reader to \cite[Section 4]{LX2025}.

\begin{lem}\label{prop-wi2}
Let $\bw_{i,2}$ be the weak solutions of \eqref{eq-wi2}, $i=1,\dots,6$. Then for sufficiently small  $\varepsilon>0$, we have
\begin{equation*}
|\nabla\bw_{i,2}(\bx)|\leq C\big(\|\bw_{i,2}\|_{L^{2}(\Omega)}+\|{\bf f}_i\|_{L^{\infty}(\partial B_R)}\big),\quad \bx\in\Omega,
\end{equation*}
where $C>0$ is a constant independently of $\varepsilon$.
\end{lem}

\begin{proof}
By using the Sobolev embedding theorem and classical $W^{2,q}$-estimate for elliptic systems, for $q>2$, we obtain 
\begin{align*}
\|\nabla \bw_{i,2}\|_{L^\infty(\Omega\setminus\Omega_{2R_0})}&\leq C\|\bw_{i,2}\|_{W^{2,q}(\Omega\setminus\Omega_{2R_0})}\nonumber\\
&\leq C\big(\|\bw_{i,2}\|_{L^{2}(\Omega\setminus\Omega_{R_0})}+\|{\bf f}_i\|_{L^{\infty}(\partial B_R)}\big).
\end{align*}
This implies that we only need to consider 
\begin{equation}\label{eq:bddu12}
\begin{cases}
\mathcal{L}_{ {\lambda}, {\mu}}\bw_{i,2}=0\quad &\text{in} \quad \Omega_{2R_0},\\
\bw_{i,2}=0\quad &\text{on} \quad \Gamma_+^{2R_0}\cup\Gamma_-^{2R_0},
\end{cases}
\end{equation}
and estimate $\nabla\bw_{i,2}$ in the narrow region $\Omega_{2R_0}$. Denote 
\begin{equation*}
\Gamma_+^{2R_0}:=\left\{({\bf x}_1,{\bf x}_2)\in\mathbb R^2:  ~{\bf x}_2=\frac{\varepsilon}{2}+\mathcal{H}_{1}({\bf x}_1),~|{\bf x}_1|\leq 2R_0\right\}
\end{equation*}
and 
\begin{equation*}
\Gamma_-^{2R_0}:=\left\{({\bf x}_1,{\bf x}_2)\in\mathbb R^2:  ~{\bf x}_2=-\frac{\varepsilon}{2}+\mathcal{H}_{2}({\bf x}_1),~|{\bf x}_1|\leq 2R_0\right\}.
\end{equation*}
Multiplying the equation in \eqref{eq:bddu12} by $\bw_{i,2}$ and integrating by parts, we have 
\begin{align*}
&\int_{\Omega_{2R_0}}\left(\mathbb{C}e(\bw_{i,2}),e(\bw_{i,2})\right)\\
&=\int_{\substack{|{\bf x}_1|=2R_0\\
					-\frac{\varepsilon}{2}+\mathcal{H}_{2}({\bf x}_1)<{\bf x}_2<\frac{\varepsilon}{2}+\mathcal{H}_{1}({\bf x}_1)}}\bw_{i,2}\big(\lambda(\nabla\cdot \bw_{i,2})+\mu(\nabla \bw_{i,2}+(\nabla\bw_{i,2})^{t})\big)\frac{\bx}{r}\\
&\leq C\int_{\substack{|{\bf x}_1|=2R_0\\
					-\frac{\varepsilon}{2}+\mathcal{H}_{2}({\bf x}_1)<{\bf x}_2<\frac{\varepsilon}{2}+\mathcal{H}_{1}({\bf x}_1)}}(|\bw_{i,2}|^2+|\nabla \bw_{i,2}|^2).
\end{align*}
By using \cite[(4.18)]{BLL2015}, we have 
\begin{align*}
\int_{\Omega_{2R_0}}|\nabla\bw_{i,2}|^2\leq C\int_{\Omega_{2R_0}}|e(\bw_{i,2})|^2\leq C\int_{\Omega_{2R_0}}\left(\mathbb{C}e(\bw_{i,2}),e(\bw_{i,2})\right).
\end{align*}
Thus, we obtain
\begin{align*}
\int_{\Omega_{2R_0}}|\nabla\bw_{i,2}|^2\leq C\int_{\substack{|{\bf x}_1|=2R_0\\
					-\frac{\varepsilon}{2}+\mathcal{H}_{2}({\bf x}_1)<{\bf x}_2<\frac{\varepsilon}{2}+\mathcal{H}_{1}({\bf x}_1)}}(|\bw_{i,2}|^2+|\nabla \bw_{i,2}|^2).
\end{align*}
Recalling that $\bw_{i,2}=0$ on $\partial D_1$, we have, for any $\bx\in\Omega_{3R_0}\setminus\Omega_{3R_0/2}$,
\begin{align*}
|\bw_{i,2}({\bf x}_1,{\bf x}_2)|&=|\bw_{i,2}({\bf x}_1,{\bf x}_2)-\bw_{i,2}({\bf x}_1,\frac{\varepsilon}{2}+\mathcal{H}_{1}({\bf x}_1))|\\
&\leq C(\varepsilon+\kappa|{\bf x}_1|^2)\|\nabla \bw_{i,2}\|_{L^\infty(\Omega_{3R_0}\setminus\Omega_{3R_0/2})}\\
&\leq C\big(\|\bw_{i,2}\|_{L^{2}(\Omega\setminus\Omega_{R_0})}+\|{\bf f}_i\|_{L^{\infty}(\partial B_R)}\big),
\end{align*}
where we used the Sobolev embedding theorem and classical $W^{2,q}$-estimate for elliptic systems in the last inequality. Hence, we derive 
\begin{align*}
\int_{\Omega_{2R_0}}|\nabla\bw_{i,2}|^2&\leq C\int_{\substack{|{\bf x}_1|=2R_0\nonumber\\
					-\frac{\varepsilon}{2}+\mathcal{H}_{2}({\bf x}_1)<{\bf x}_2<\frac{\varepsilon}{2}+\mathcal{H}_{1}({\bf x}_1)}}(|\bw_{i,2}|^2+|\nabla \bw_{i,2}|^2)\\
                    &\leq C\big(\|\bw_{i,2}\|_{L^{2}(\Omega\setminus\Omega_{R_0})}^2+\|{\bf f}_i\|_{L^{\infty}(\partial B_R)}^2\big).
\end{align*}
With this inequality in hand, by adapting the iterative arguments used in the proof of \cite[Propositions 5.1 and 5.2]{DLX2025}, we derive 
\begin{equation*}
|\nabla\bw_{i,2}(\bx)|\leq C\big(\|\bw_{i,2}\|_{L^2(\Omega)}+\|{\bf f}_{i}\|_{L^{\infty}(\partial B_R)}\big),\quad\bx\in\Omega_{2R_0}.
\end{equation*}
The proof is finished.
\end{proof}

Combining \eqref{decom-wi} and Lemmas \ref{prop-wi1}--\ref{prop-wi2}, we derive the gradient estimates of $\bw_i$ as follows.
\begin{prop}\label{thm-Dwi}
Let $\bw_{i}$ be the weak solutions of \eqref{eq-wi}, $i=1,\dots,6$. Then for sufficiently small  $\varepsilon>0$, we have
\begin{equation*}
\nabla\bw_{i}=\nabla\bv_i+\Ocal(1),
\end{equation*}
where $\bv_i$ are given in \eqref{auxiliary improved}.
\end{prop}

Now, we are in a position to prove  Proposition \ref{thm-Aij}.

\begin{proof}[Proof of Proposition \ref{thm-Aij}]
{\bf Step 1: Proof of \eqref{diag-2}.}
Define 
\begin{equation*}
\mathcal{C}_{R_0}:=\left\{({\bf x}_1,{\bf x}_2)\in \mathbb{R}^{2}: -\frac{\varepsilon}{2}+2\min_{|{\bf x}_1|=R_0}\mathcal{H}_{2}({\bf x}_1)<{\bf x}_2<\frac{\varepsilon}{2}+2\max_{|{\bf x}_1|=R_0}\mathcal{H}_{1}({\bf x}_1),~|{\bf x}_1|<R_0\right\}.
\end{equation*}
Then combining with the gradient estimates for elliptic systems, we have 
\begin{align}\label{est-alp111}
\balpha_{11,1}=\int_{\partial D}\frac{\partial\bw_1}{\partial{\bnu}}\Bigg|_{+}{\bm{\vartheta}}_1&=\int_{\partial D_1\cap\mathcal{C}_{R_0}}\frac{\partial\bw_1}{\partial{\bnu}}\Bigg|_{+}{\bm{\vartheta}}_1+\int_{\partial D_1\setminus\mathcal{C}_{R_0}}\frac{\partial\bw_1}{\partial{\bnu}}\Bigg|_{+}{\bm{\vartheta}}_1\nonumber\\
&=\int_{\partial D_1\cap\mathcal{C}_{R_0}}\frac{\partial\bw_1}{\partial{\bnu}}\Bigg|_{+}{\bm{\vartheta}}_1+\Ocal(1).
\end{align}
By using 
\begin{align}\label{def-nu-wi}
\frac{\partial \bw_{i}}{\partial\bnu}=\lambda(\nabla\cdot \bw_{i}){\bf n}+\mu(\nabla \bw_{i}+(\nabla\bw_{i})^{t}){\bf n},
\end{align}
where 
\begin{equation*}
{\bf n}=(n_1,n_2),\quad n_1=\frac{\partial_{{\bf x}_1}\mathcal{H}_{1}({\bf x}_1)}{\sqrt{1+|\partial_{{\bf x}_1}\mathcal{H}_{1}({\bf x}_1)|^2}},\quad n_2=-\frac{1}{\sqrt{1+|\partial_{{\bf x}_1}\mathcal{H}_{1}({\bf x}_1)|^2}},
\end{equation*}
we have 
\begin{align}\label{int-w1}
&\int_{\partial D_1\cap\mathcal{C}_{R_0}}\frac{\partial\bw_1}{\partial{\bnu}}\Bigg|_{+}{\bm{\vartheta}}_1\nonumber\\
&=\int_{\partial D_1\cap\mathcal{C}_{R_0}}\left(\lambda(\partial_{{\bf x}_1}w_{1}^1+\partial_{{\bf x}_2}w_{1}^2)n_1+\mu(2n_1\partial_{{\bf x}_1}w_{1}^1+n_2(\partial_{{\bf x}_1}w_{1}^2+\partial_{{\bf x}_2}w_{1}^1))\right).
\end{align}
In view of Proposition \ref{thm-Dwi} and \eqref{auxiliary improved}, we find that the biggest term on the right-hand side of \eqref{int-w1} is 
$$\mu\int_{\partial D_1\cap\mathcal{C}_{R_0}}n_2\partial_{{\bf x}_2}w_{1}^1,$$
and the remaining terms are bounded by $\Ocal(1)$. By using a direct calculation together with \eqref{delta_x'}, we obtain
\begin{align*}
\mu\int_{\partial D_1\cap\mathcal{C}_{R_0}}n_2\partial_{{\bf x}_2}w_{1}^1&=-\mu\int_{|{\bf x}_1|\leq R_0}\frac{1}{\delta({\bf x}_1)}\ d{\bf x}_1+\Ocal(1)\\
&=-\mu\int_{|x_1|\leq R_0}\frac{1}{\varepsilon+\kappa|x_1|^2}\ dx_1+\varepsilon^{\frac{\alpha-1}{2}}\Ocal(1)\\
		&=-\frac{2\mu}{\sqrt{\kappa\varepsilon}}\int_{0}^{R_0\sqrt{\frac{\kappa}{\varepsilon}}}\frac{1}{1+r^2}\ dr+\varepsilon^{\frac{\alpha-1}{2}}\Ocal(1)=-\frac{\mu\pi}{\sqrt{\kappa\varepsilon}}+\varepsilon^{\frac{\alpha-1}{2}}\Ocal(1).
\end{align*}
Thus, we obtain
\begin{equation*}
\balpha_{11,1}=-\frac{\mu\pi}{\sqrt{\kappa\varepsilon}}+\varepsilon^{\frac{\alpha-1}{2}}\Ocal(1).
\end{equation*}
Similar to \eqref{est-alp111}, we derive 
\begin{equation*}
\balpha_{12,1}=\int_{\partial D}\frac{\partial\bw_4}{\partial{\bnu}}\Bigg|_{+}{\bm{\vartheta}}_1=\int_{\partial D_1\cap\mathcal{C}_{R_0}}\frac{\partial\bw_4}{\partial{\bnu}}\Bigg|_{+}{\bm{\vartheta}}_1+\Ocal(1).
\end{equation*}
By replicating the proof as above, we get that the leading order term of $\balpha_{12,1}$ is 
\begin{align*}
\mu\int_{\partial D_1\cap\mathcal{C}_{R_0}}n_2\partial_{{\bf x}_2}w_{4}^1&=\mu\int_{|{\bf x}_1|\leq R_0}\frac{1}{\delta({\bf x}_1)}\ d{\bf x}_1+\Ocal(1)\\
&=\frac{\mu\pi}{\sqrt{\kappa\varepsilon}}+\varepsilon^{\frac{\alpha-1}{2}}\Ocal(1).
\end{align*}
This gives 
$$\balpha_{12,1}=\frac{\mu\pi}{\sqrt{\kappa\varepsilon}}+\varepsilon^{\frac{\alpha-1}{2}}\Ocal(1).$$
The estimates of $\balpha_{21,2}$ and $\balpha_{22,2}$ are proved in a similar manner, we thus omit the details.

{\bf Step 2: Proof of \eqref{est-3132} and \eqref{est-311}.} 
Note that 
\begin{align*}
\beta_{3k}=\int_{\partial D_k}\frac{\partial\bw_3}{\partial{\bnu}}\Bigg|_{+}{\bm{\vartheta}}_3=\int_{\partial D_k\cap\mathcal{C}_{R_0}}\frac{\partial\bw_3}{\partial{\bnu}}\Bigg|_{+}{\bm{\vartheta}}_3+\Ocal(1),\quad k=1,2.
\end{align*}
It follows from \eqref{auxiliary improved} that 
\begin{align*}
|\partial_{\bx_1}\bv_{3}^1|\leq C|\bx_1|,\quad|\partial_{\bx_2}\bv_{3}^1|\leq C,\quad |\partial_{\bx_1}\bv_{3}^2|\leq C,\quad\partial_{x_2}\bv_{3}^2=-\frac{\bx _1}{\delta(\bx_1)}.
\end{align*}
Then by using Proposition \ref{thm-Dwi} and  \eqref{def-nu-wi}, we obtain
\begin{align*}
\beta_{3k}&=\int_{\partial D_k\cap\mathcal{C}_{R_0}}\Big(\lambda(\partial_{{\bf x}_1}w_{3}^1+\partial_{{\bf x}_2}w_{3}^2)(n_1\bx_2-n_2\bx_1)\\
&\quad+\mu\big(\bx_2(2n_1\partial_{{\bf x}_1}w_{3}^1+n_2(\partial_{{\bf x}_1}w_{3}^2+\partial_{{\bf x}_2}w_{1}^1))-\bx_1(n_1(\partial_{{\bf x}_1}w_{3}^2+\partial_{{\bf x}_2}w_{1}^1)+2n_2\partial_{{\bf x}_2}w_{3}^2)\big)\Big)\\
&=\Ocal(1),\quad k=1,2.
\end{align*}
Similarly, we obtain $\balpha_{31,1}=\beta_{11}=\Ocal(1)$ and $\balpha_{32,1}=-\beta_{12}=\Ocal(1)$.

It follows from \eqref{def-Aij} and \eqref{def:alpha} that
\begin{align}\label{beta3132}
\beta_{31}+\beta_{32}=A_{33}+A_{36}=\int_{\partial D_1}\frac{\partial(\bw_3+\bw_6)}{\partial{\bnu}}\Bigg|_{+}\bvarkappa_3=\int_{\partial D_2}\frac{\partial(\bw_3+\bw_6)}{\partial{\bnu}}\Bigg|_{+}\bvarkappa_3.
\end{align}
By using \eqref{eq-wi} and $\bzeta_6(-\bx)=-\bzeta_3(\bx)$, we obtain that $\bw_3+\bw_6$ satisfies 
\begin{align}\label{eq-w3w6}
\begin{cases}
\mathcal{L}_{ {\lambda}, {\mu}}({\bf w}_3+{\bf w}_6)=0,  &  \bx\in D^{e},\\
{\bf w}_3+{\bf w}_6={\bm{\vartheta}}_3+{\bm{\vartheta}}_6,    & \bx\in\partial D,\\
{\bf w}_3+{\bf w}_6=\Ocal(|\bx|^{-1}),& |\bx|\rightarrow\infty.
\end{cases}
\end{align}
Multiplying the equation in \eqref{eq-w3w6} by ${\bf w}_3+{\bf w}_6$ and integrating by parts in $B_R\setminus\overline{D}$, we have 
\begin{align*}
2(\beta_{31}+\beta_{32})=-\int_{B_R\setminus\overline{D}}\left(\mathbb{C}e(\bw_3+\bw_6),e(\bw_3+\bw_6)\right)+\int_{\partial B_R}(\bw_3+\bw_6)\partial_{\bnu}(\bw_3+\bw_6).
\end{align*}
As $|\bx|\rightarrow\infty$, 
\begin{equation}\label{R-condition}
{\bf w}_3+{\bf w}_6=\Ocal(|\bx|^{-1})\quad\mbox{and}\quad \partial_{\bnu}({\bf w}_3+{\bf w}_6)=\Ocal(|\bx|^{-2}).
\end{equation}
This yields 
\begin{align}\label{int-w3+w6}
\int_{\partial B_R}(\bw_3+\bw_6)\partial_{\bnu}(\bw_3+\bw_6)\rightarrow 0,\quad\mbox{as}~R\rightarrow\infty,
\end{align}
and thus, 
\begin{equation*}
2(\beta_{31}+\beta_{32})=-\int_{D^e}\left(\mathbb{C}e(\bw_3+\bw_6),e(\bw_3+\bw_6)\right)<0.
\end{equation*}

In view of \eqref{def:alpha} and \eqref{def-Aij} again, we obtain
\begin{align}\label{beta31-32}
\beta_{31}-\beta_{32}=A_{33}-A_{36}=\int_{\partial D_1}\frac{\partial(\bw_3-\bw_6)}{\partial{\bnu}}\Bigg|_{+}{\bm{\kappa}}_3=-\int_{\partial D_2}\frac{\partial(\bw_3-\bw_6)}{\partial{\bnu}}\Bigg|_{+}{\bm{\kappa}}_3.
\end{align}
Analogously to \eqref{eq-w3w6}, we have 
\begin{align}\label{eq-w3-w6}
\begin{cases}
\mathcal{L}_{ {\lambda}, {\mu}}({\bf w}_3-{\bf w}_6)=0,  &  \bx\in D^{e},\\
{\bf w}_3-{\bf w}_6={\bm{\vartheta}}_3-{\bm{\vartheta}}_6,    & \bx\in\partial D,\\
{\bf w}_3-{\bf w}_6={\bf f}_3-{\bf f}_6,& \bx\in\partial B_R.
\end{cases}
\end{align}
Multiplying the equation in \eqref{eq-w3-w6} by ${\bf w}_3-{\bf w}_6$ and ${\bf w}_3+{\bf w}_6$, respectively, and integrating by parts in $B_R\setminus\overline{D}$ yields 
\begin{align}\label{form-beta3132}
2(\beta_{31}-\beta_{32})=-\int_{B_R\setminus\overline{D}}\left(\mathbb{C}e(\bw_3-\bw_6),e(\bw_3-\bw_6)\right)+\int_{\partial B_R}(\bw_3-\bw_6)\partial_{\bnu}(\bw_3-\bw_6)
\end{align}
and 
\begin{equation}\label{w3-w6-w36}
\int_{B_R\setminus\overline{D}}\left(\mathbb{C}e(\bw_3-\bw_6),e(\bw_3+\bw_6)\right)-\int_{\partial B_R}(\bw_3+\bw_6)\partial_{\bnu}(\bw_3-\bw_6)=0,
\end{equation}
where we used \eqref{beta31-32} in the second equality. 
Similarly, multiplying the equation in \eqref{eq-w3w6} by ${\bf w}_3-{\bf w}_6$, integrating by parts in $B_R\setminus\overline{D}$, and using \eqref{beta3132},  we have 
\begin{equation}\label{w3w6-w36}
\int_{B_R\setminus\overline{D}}\left(\mathbb{C}e(\bw_3+\bw_6),e(\bw_3-\bw_6)\right)-\int_{\partial B_R}(\bw_3-\bw_6)\partial_{\bnu}(\bw_3+\bw_6)=0.
\end{equation}
Applying  \eqref{R-condition} gives 
\begin{align}\label{w3-w6infty}
\int_{\partial B_R}(\bw_3-\bw_6)\partial_{\bnu}(\bw_3+\bw_6)\rightarrow 0\quad\mbox{as}~R\rightarrow\infty.
\end{align}
Since $(\mathbb{C}A,B)=(A,\mathbb{C}B)$ for any $2\times2$ matrices $A$ and $B$ (see, for instance, \cite{BLL2015}), it follows from \eqref{w3-w6-w36}, \eqref{w3w6-w36}, and \eqref{w3-w6infty}  that 
\begin{align*}
\int_{\partial B_R}(\bw_3+\bw_6)\partial_{\bnu}(\bw_3-\bw_6)\rightarrow 0\quad\mbox{as}~R\rightarrow\infty.
\end{align*}
This together with \eqref{w3-w6infty} and \eqref{int-w3+w6} implies that as $R\rightarrow\infty$, the term $\int_{\partial B_R}(\bw_3-\bw_6)\partial_{\bnu}(\bw_3-\bw_6)$ in \eqref{form-beta3132} has the same asymptotics as $4\int_{\partial B_R}\bw_3\partial_{\bnu}\bw_3$. Since the first term on the right-hand side of \eqref{form-beta3132} is negative, it suffices to analyze $\int_{\partial B_R}\bw_3\partial_{\bnu}\bw_3$ as $R\rightarrow\infty$.

Recalling that
\begin{align}\label{def-w3}
{\bf w}_3:=\hat{\bS}_D^\omega[\bzeta_3]=\int_{\partial D}(\bGa^{0,1}+ \bGa^{0,2})(\bx-\by)\bzeta_3(\by)ds(\by).
\end{align}
It follows from \eqref{def:ga01} and \eqref{def:ga2} that $\partial_{\bnu}\bGa^{0,1}=0$ and for $|\bx| \gg 1$ and $\by\in \partial D$, we have 
\begin{equation}\label{gamma02}
\bGa^{0,2}(\bx-\by)=\frac{1}{4\pi}\Big(\frac{1}{\mu}+\frac{1}{\lambda +2\mu}\Big)\ln\abs{\bx}\bI-\frac{1}{4\pi}\Big(\frac{1}{\mu}-\frac{1}{\lambda +2\mu}\Big)\frac{\bx\otimes\bx}{|\bx|^2}+\Ocal(|\bx|^{-1}).
\end{equation}
By using \eqref{def-nu-wi}, we obtain on $\partial B_R$,
\begin{align*}
\partial_{\bnu}(\ln\abs{\bx}\bI\bzeta_3(\by))=\frac{\mu}{R}\bzeta_3+\frac{\lambda+\mu}{R^3}(\bzeta_3\cdot\bx)\bx
\end{align*}
and 
\begin{align*}
\partial_{\bnu}\left(\frac{\bx\otimes\bx}{|\bx|^2}\bzeta_3(\by)\right)=\lambda \frac{(\bzeta_3\cdot\bx)\bx}{R^3}+\mu\left(\frac{\bzeta_3}{R}-\frac{(\bzeta_3\cdot\bx)\bx}{R^3}\right).
\end{align*}
Thus, from \eqref{def-w3} and \eqref{gamma02}, we obtain
\begin{align*}
\partial_{\bnu}\bw_3&=a\int_{\partial D}\partial_{\bnu}(\ln\abs{\bx}\bI\bzeta_3(\by))-b\int_{\partial D}\partial_{\bnu}\Big(\frac{\bx\otimes\bx}{|\bx|^2}\bzeta_3(\by)\Big)+o(1)\nonumber\\
&=\frac{\mu}{2\pi(\lambda+2\mu)R}I_{\bzeta_3}+\frac{1}{2\pi R^3}\Big(1+\frac{\lambda}{\lambda+2\mu}\Big)(I_{\bzeta_3}\cdot\bx)\bx,
\end{align*}
where 
\begin{equation*}
a:=\frac{1}{4\pi}\Big(\frac{1}{\mu}+\frac{1}{\lambda +2\mu}\Big),\quad b:=\frac{1}{4\pi}\Big(\frac{1}{\mu}-\frac{1}{\lambda +2\mu}\Big),
\end{equation*}
and
\begin{equation*}
I_{\bzeta_3}=(I_{\bzeta_3,1},I_{\bzeta_3,2}):=\int_{\partial D}\bzeta_3(\by)ds(\by).
\end{equation*}
Hence, we derive 
\begin{align}\label{int-w3-Dw3}
&\int_{\partial B_R}\bw_3\partial_{\bnu}\bw_3\nonumber\\
&=\hat{\eta}_{\omega}\int_{\partial B_R} I_{\bzeta_3}\Big(\frac{\mu}{2\pi(\lambda+2\mu)R}I_{\bzeta_3}+\frac{1}{2\pi R^3}\big(1+\frac{\lambda}{\lambda+2\mu}\big)(I_{\bzeta_3}\cdot\bx)\bx\Big)\nonumber\\
&\quad+a\int_{\partial B_R} (I_{\bzeta_3}\ln R)\Big(\frac{\mu}{2\pi(\lambda+2\mu)R}I_{\bzeta_3}+\frac{1}{2\pi R^3}\big(1+\frac{\lambda}{\lambda+2\mu}\big)(I_{\bzeta_3}\cdot\bx)\bx\Big)\nonumber\\
&\quad-\frac{b}{R^2}\int_{\partial B_R} (I_{\bzeta_3}\bx\otimes\bx)\Big(\frac{\mu}{2\pi(\lambda+2\mu)R}I_{\bzeta_3}+\frac{1}{2\pi R^3}\big(1+\frac{\lambda}{\lambda+2\mu}\big)(I_{\bzeta_3}\cdot\bx)\bx\Big)\nonumber\\
&\quad+o(1)=:\mbox{I}+\mbox{II}+\mbox{III}+o(1).
\end{align}
Note that
\begin{equation*}
\int_{\partial B_R}(I_{\bzeta_3}\cdot\bx)^2=\int_0^{2\pi}(I_{\bzeta_3,1}R\cos\theta+I_{\bzeta_3,2}R\sin\theta)^2R\ d\theta=\pi R^3|I_{\bzeta_3}|^2,
\end{equation*}
which yields
\begin{align}\label{est-main-term}
\mbox{I}&=\frac{\mu}{\lambda+2\mu}\hat{\eta}_{\omega}|I_{\bzeta_3}|^2+\frac{1}{2}\Big(1+\frac{\lambda}{\lambda+2\mu}\Big)\hat{\eta}_{\omega}|I_{\bzeta_3}|^2=\hat{\eta}_{\omega}|I_{\bzeta_3}|^2
\end{align}
and
\begin{align*}
\mbox{II}=\frac{1}{4\pi}\Big(\frac{1}{\mu}+\frac{1}{\lambda +2\mu}\Big)|I_{\bzeta_3}|^2\ln R.
\end{align*}
Recalling that $\beta_{3k}=\Ocal(1)$ with $k=1,2$, we note that $\beta_{3k}$ should be independent of $R$. Hence, the term  $\mbox{II}$ can be absorbed into the first term on the right-hand side of \eqref{form-beta3132}.
Direct calculations give 
\begin{align}\label{est-temrIII}
\mbox{III}&=-\frac{b}{R^3}\Big(\frac{\mu}{2\pi(\lambda+2\mu)}+\frac{1}{2\pi}\big(1+\frac{\lambda}{\lambda+2\mu}\big)\Big)\int_{\partial B_R}(I_{\bzeta_3}\cdot\bx)^2\nonumber\\
&=-\frac{1}{8\pi}\Big(\frac{1}{\mu}-\frac{1}{\lambda +2\mu}\Big)\cdot\Big(1+\frac{\lambda+\mu}{\lambda+2\mu}\big)\Big)|I_{\bzeta_3}|^2.
\end{align}
Therefore, for any fixed $|\omega|\ll 1$, as $R\rightarrow\infty$, we derive from \eqref{form-beta3132}, \eqref{int-w3-Dw3}, \eqref{est-main-term}, and \eqref{est-temrIII} that
\begin{align}\label{diff-beta3132}
2(\beta_{31}-\beta_{32})=\hat{\eta}_{\omega}|I_{\bzeta_3}|^2-\frac{1}{8\pi}\Big(\frac{1}{\mu}-\frac{1}{\lambda +2\mu}\Big)\cdot\Big(1+\frac{\lambda+\mu}{\lambda+2\mu}\big)\Big)|I_{\bzeta_3}|^2+\Ocal(1),
\end{align}
where $\Ocal(1)$ is independent of $\omega$.

We next analyze $\hat{\eta}_{\omega}$. From \eqref{eq:defheta} and \eqref{pa:ksp}, it follows that
\begin{align*}
\Re\hat{\eta}_{\omega}&=\frac{(\lambda+\mu)(16\pi\Re c_1-1)}{8\mu(\lambda+2\mu)\pi}+\frac{\Re\tilde{\eta}}{\mu}+\frac{\ln|\omega|-\ln{c_s}}{4\mu\pi}-\frac{\ln|\omega|-\ln{c_p}}{4(\lambda+2\mu)\pi}\\
&=\frac{(\lambda+\mu)(16\pi\Re c_1-1)}{8\mu(\lambda+2\mu)\pi}+\frac{\Re\tilde{\eta}}{\mu}+\frac{(\lambda+\mu)\ln|\omega|}{4\mu(\lambda+2\mu)\pi}-\frac{\ln{c_s}}{4\mu\pi}+\frac{\ln{c_p}}{4(\lambda+2\mu)\pi},
\end{align*}
and 
\begin{align*}
\Im\hat{\eta}_{\omega}&=\frac{16\pi(\lambda+\mu)\Im c_1}{8\mu(\lambda+2\mu)\pi}+\frac{\Im\tilde{\eta}}{\mu}+\frac{(\lambda+\mu)\theta}{4\mu(\lambda+2\mu)\pi}.
\end{align*}
Here, we used $\omega=|\omega|e^{\ri\theta}$ with $\theta=\arg\omega$. 
It is well known established that resonant frequencies are located in the lower half of the complex plane \cite{Dyatlov2019}. Consequently, applying the conditions $|\omega|\ll 1$ and $\Re\omega> 0$ from Definition \ref{def-frequency}, we deduce that $\theta\in(-\pi/2,0)$.
Therefore, 
$$\Re\hat{\eta}_{\omega}<0,\quad \Im\hat{\eta}_{\omega}<0.$$
Coming back to \eqref{diff-beta3132},  we derive $\Re(\beta_{31}-\beta_{32})<0$ and $\Im(\beta_{31}-\beta_{32})<0$ when $|\omega|\ll 1$. This completes the proof of Proposition \ref{thm-Aij}.
\end{proof}

\subsection{Gradient estimates of eigenmodes}\label{sec-eigenmodes}

It follows from \eqref{def-eigenmode} and \eqref{def-Acal0} that, the eigenmodes are given by
\begin{equation}\label{def-eigen}
{\bu}_i=\hat{\bS}_{D}^{\omega_i}[\bpsi_i]+\Ocal(\omega_i^2\ln{\omega_i}+\delta),\quad i=1,\dots,6,
\end{equation}
where $\bpsi_i$ and $\omega_i$ are given in Theorem \ref{thm:res}. Combining with \eqref{def:xi}, we derive 
\begin{align}\label{def-u1}
{\bu}_1(\bx)=
\begin{cases}
\bvarkappa_2+\Ocal(\omega_1^2\ln{\omega_1}),&\quad\bx\in\partial D_1,\\
-\bvarkappa_2+\Ocal(\omega_1^2\ln{\omega_1}),&\quad\bx\in\partial D_2,
\end{cases}
\end{align}
\begin{align}\label{def-u2}
{\bu}_2(\bx)=
\begin{cases}
\bvarkappa_2+\Ocal(\omega_{2,1}^2\ln{\omega_{2,1}}),&\quad\bx\in\partial D_1,\\
\bvarkappa_2+\Ocal(\omega_{2,1}^2\ln{\omega_{2,1}}),&\quad\bx\in\partial D_2,
\end{cases}
\end{align}
for $i=3,4$,
\begin{align}\label{def-u3u4}
{\bu}_i(\bx)=
\begin{cases}
e_{i,1}\bvarkappa_1+e_{i,3}\bvarkappa_3+\Ocal(\omega_{i}^2\ln{\omega_{i}}),&\quad\bx\in\partial D_1,\\
e_{i,1}\bvarkappa_1+e_{i,6}\bvarkappa_3+\Ocal(\omega_{i}^2\ln{\omega_{i}}),&\quad\bx\in\partial D_2,
\end{cases}
\end{align}
and for $i=5,6$,
\begin{align}\label{def-u5u6}
{\bu}_i(\bx)=
\begin{cases}
e_{i,1}\bvarkappa_1+e_{i,3}\bvarkappa_3+\Ocal(\omega_{i}^2\ln{\omega_{i}}),&\quad\bx\in\partial D_1,\\
-e_{i,1}\bvarkappa_1+e_{i,3}\bvarkappa_3+\Ocal(\omega_{i}^2\ln{\omega_{i}}),&\quad\bx\in\partial D_2,
\end{cases}
\end{align}
where $e_{i,1}$, $e_{i,3}$, and $e_{i,6}$ are given in \eqref{def-e3}, \eqref{def-e4}, \eqref{def-e5}, and \eqref{def-e6}.

Building on the antiphase oscillations in certain components of the resonant modes ${\bu}_i$ (for $i=1,3,4,5,6$) between the two resonators, as seen from \eqref{def-u1}--\eqref{def-u5u6}, the gradient of the eigenmodes may exhibit singular behavior as the inter-resonator distance $\varepsilon$ approaches zero. In the following analysis, we quantify the blow-up rates of these gradients under the condition that the contrast parameter $\delta>0$ is sufficiently small.

In view of \eqref{def-u1}--\eqref{def-u5u6}, and \eqref{eq-wi}, we find that 
\begin{equation*}
{\bu}_1={\bw}_2-{\bw}_5+\Ocal(\omega_1^2\ln{\omega_1}),\quad {\bu}_2={\bw}_2+{\bw}_5+\Ocal(\omega_{2,1}^2\ln{\omega_{2,1}}),
\end{equation*}
\begin{equation*}
{\bu}_i=e_{i,1}({\bw}_1+{\bw}_4)+e_{i,3}{\bw}_3+e_{i,6}{\bw}_6+\Ocal(\omega_i^2\ln{\omega_i}),\quad i=3,4,
\end{equation*}
and 
\begin{equation*}
{\bu}_i=e_{i,1}({\bw}_1-{\bw}_4)+e_{i,3}({\bw}_3+{\bw}_6)+\Ocal(\omega_i^2\ln{\omega_i}),\quad i=5,6,
\end{equation*}
and they satisfy, respectively,
\begin{align*}
\begin{cases}
\mathcal{L}_{ {\lambda}, {\mu}}{\bf u}_1+\rho\omega_1^2{\bf u}_1=0,  &  \bx\in D^{e},\\
{\bf u}_1={\bm{\kappa}}_2+{\bf f}_{1,1},    & \bx\in\partial D_1,\\
{\bf u}_1=-{\bm{\kappa}}_2+{\bf f}_{1,2},    & \bx\in\partial D_2,\\
{\bf u}_1={\bf g}_1,& \bx\in\partial B_R,
\end{cases}
\quad
\begin{cases}
\mathcal{L}_{ {\lambda}, {\mu}}{\bf u}_2+\rho\omega_2^2{\bf u}_2=0,  &  \bx\in D^{e},\\
{\bf u}_2={\bm{\kappa}}_2+{\bf f}_{2,1},    & \bx\in\partial D_1,\\
{\bf u}_2={\bm{\kappa}}_2+{\bf f}_{2,2},    & \bx\in\partial D_2,\\
{\bf u}_2={\bf g}_2,& \bx\in\partial B_R,
\end{cases}
\end{align*}
\begin{align*}
\begin{cases}
\mathcal{L}_{ {\lambda}, {\mu}}{\bf u}_i+\rho\omega_i^2{\bf u}_i=0,  &  \bx\in D^{e},\\
{\bf u}_i=e_{i,1}{\bm{\kappa}}_1+e_{i,3}{\bm{\kappa}}_3+{\bf f}_{i,1},    & \bx\in\partial D_1,\\
{\bf u}_i=e_{i,1}{\bm{\kappa}}_1+e_{i,6}{\bm{\kappa}}_3+{\bf f}_{i,2},    & \bx\in\partial D_2,\\
{\bf u}_i={\bf g}_i,& \bx\in\partial B_R,\quad i=3,4,
\end{cases}
\end{align*}
and 
\begin{align*}
\begin{cases}
\mathcal{L}_{ {\lambda}, {\mu}}{\bf u}_i+\rho\omega_i^2{\bf u}_i=0,  &  \bx\in D^{e},\\
{\bf u}_i=e_{i,1}{\bm{\kappa}}_1+e_{i,3}{\bm{\kappa}}_3+{\bf f}_{i,1},    & \bx\in\partial D_1,\\
{\bf u}_i=-e_{i,1}{\bm{\kappa}}_1+e_{i,3}{\bm{\kappa}}_3+{\bf f}_{i,2},    & \bx\in\partial D_2,\\
{\bf u}_i={\bf g}_i,& \bx\in\partial B_R,\quad i=5,6,
\end{cases}
\end{align*}
where ${\bf f}_{i,j}=\Ocal(\omega_i^2\ln{\omega_i})$ with $i=1,3,4,5,6$, ${\bf f}_{2,j}=\Ocal(\omega_{2,1}^2\ln{\omega_{2,1}})$, and ${\bf g}_i(\bx)=\hat{\bS}_{D}^{\omega_i}[\bpsi_i]+\Ocal(\omega_i^2\ln{\omega_i}+\delta)$, $i=1,\dots,6$, $j=1,2$.

\begin{thm}\label{thm-eigonmode}
Consider the system \eqref{eq:or2} and let the eigenmodes be given in \eqref{def-eigen}. 
If we choose $\varepsilon=\Ocal(\delta^{\beta})$ with $0<\beta<2$, then as $\delta\rightarrow0$, for $\bx\in\mathbb R^2\setminus\overline{D}$, it holds that,

(1) for $i=1,5,6$,
\begin{align*}
|\nabla{\bf u}_{i}(\bx)|&\leq \frac{C}{\varepsilon+\kappa|{\bf x}_1|^2}+C\big(\|\bu_{i}\|_{L^2(\mathbb R^2\setminus\overline{D})}+\|{\bf g}_{i}\|_{L^{\infty}(\partial B_R)}+\|{\bf f}_{i,1}\|_{C^2(\partial D_1)}\\
&\quad+\|{\bf f}_{i,2}\|_{C^2(\partial D_2)}\big),\\
|\nabla{\bf u}_{i}(0,{\bf x}_2)|&\geq \frac{1}{C\varepsilon};
\end{align*}

(2) for $i=2$,
\begin{align*}
|\nabla{\bf u}_{2}(\bx)|&\leq \frac{C|{\bf f}_{2,1}({\bf x}_1,\varepsilon/2+\mathcal{H}_1({\bf x}_1))-{\bf f}_{2,2}({\bf x}_1,-\varepsilon/2+\mathcal{H}_2({\bf x}_1))|}{\varepsilon+\kappa|{\bf x}_1|^2}\\
&\quad+C\big(\|\bu_{2}\|_{L^2(\mathbb R^2\setminus\overline{D})}+\|{\bf g}_{2}\|_{L^{\infty}(\partial B_R)}+\|{\bf f}_{2,1}\|_{C^2(\partial D_1)}+\|{\bf f}_{2,2}\|_{C^2(\partial D_2)}\big);
\end{align*}

(3) for $i=3,4$,
\begin{align*}
|\nabla{\bf u}_{i}(\bx)|&\leq C\left(\frac{|{\bf x}_1|}{\varepsilon+\kappa|{\bf x}_1|^2}+1\right)+C\big(\|\bu_{i}\|_{L^2(\mathbb R^2\setminus\overline{D})}+\|{\bf g}_{i}\|_{L^{\infty}(\partial B_R)}+\|{\bf f}_{i,1}\|_{C^2(\partial D_1)}\\
&\quad+\|{\bf f}_{i,2}\|_{C^2(\partial D_2)}\big),\\
|\nabla{\bf u}_{i}({\bf x}_1,{\bf x}_2)|&\geq\frac{1}{C\sqrt{\varepsilon}},\quad |{\bf x}_1|=\sqrt{\kappa^{-1}\varepsilon},
\end{align*}
where $C$ is a positive constant independent of $\varepsilon$, and $\kappa$ is the curvature of $\partial D$ at $(0,\varepsilon/2)$ and $(0,-\varepsilon/2)$. 
\end{thm}

\begin{proof}
By the classical gradient estimates for elliptic systems (see, for instance, \cite{ADN1959}), we have
\begin{equation}\label{est-u11-out}
		\|\nabla \bu_{i}\|_{L^\infty(\mathbb R^2\setminus\overline B_R)}\leq C,\quad i=1,\dots,6,
	\end{equation}
where $C>0$ is a constant independently of $\varepsilon$. Thus, we shall prove the estimates of $|\nabla \bu_{i}(\bx)|$ for $\bx\in B_R\setminus\overline{D}=:\Omega$.

{\bf Step 1: Estimate of $\nabla{\bu}_i$ with $i=1,5,6$.} We provide the details for the case $i=1$ since the cases $i=5,6$ are analogous and thus omitted. Decompose ${\bu}_1$ as 
\begin{equation}\label{dec-u1}
{\bu}_1={\bu}_{1,1}+{\bu}_{1,2},
\end{equation}
where 
\begin{align*}
\begin{cases}
\mathcal{L}_{ {\lambda}, {\mu}}{\bf u}_{1,1}+\rho\omega_1^2{\bf u}_{1,1}=0,  &  \bx\in D^{e},\\
{\bf u}_{1,1}={\bm{\kappa}}_2+{\bf f}_{1,1},    & \bx\in\partial D_1,\\
{\bf u}_{1,1}=-{\bm{\kappa}}_2+{\bf f}_{1,2},    & \bx\in\partial D_2,\\
{\bf u}_{1,1}=0,& \bx\in\partial B_R,
\end{cases}
\end{align*}
and 
\begin{align*}
\begin{cases}
\mathcal{L}_{ {\lambda}, {\mu}}{\bf u}_{1,2}+\rho\omega_1^2{\bf u}_{1,2}=0,  &  \bx\in D^{e},\\
{\bf u}_{1,2}=0,    & \bx\in\partial D_1,\\
{\bf u}_{1,2}=0,    & \bx\in\partial D_2,\\
{\bf u}_{1,2}={\bf g}_1,& \bx\in\partial B_R.
\end{cases}
\end{align*}
Applying a similar argument in the proof of Lemma \ref{prop-wi2}, we obtain
\begin{equation}\label{est-Du12}
|\nabla\bu_{1,2}(\bx)|\leq C\big(\|\bu_{1,2}\|_{L^{2}(D^e)}+\|{\bf g}_1\|_{L^{\infty}(\partial B_R)}\big),\quad \bx\in D^{e}.
\end{equation}
To estimate $\nabla{\bf u}_{1,1}$, we construct an auxiliary function $\bar{\bf u}_{1,1}\in C^{2,\alpha}(\mathbb{R}^2)$, such that $\bar{\bf u}_{1,1}={\bm{\kappa}}_2+{\bf f}_{1,1}$ on $\partial D_1$, $\bar{\bf u}_{1,1}=-{\bm{\kappa}}_2+{\bf f}_{1,2}$ on $\partial D_2$, $\bar{\bf u}_{1,1}=0$ on $\partial B_R$, 
\begin{equation}\label{def-baru11}
\bar{\bf u}_{1,1}=p({\bm{\kappa}}_2+{\bf f}_{1,1})+(1-p)(-{\bm{\kappa}}_2+{\bf f}_{1,2})\quad\hbox{in}\ \Omega_{2R_0},
\end{equation}
and $\|\bar {\bf u}_{1,1}\|_{C^{2,\alpha}(D^e\setminus\Omega_{R_0})}\leq\,C$, where $p$ is defined in \eqref{def-p}. Using some calculations, we have 
\begin{equation}\label{est-baru11}
|\partial_{{\bf x}_1}\bar {\bf u}_{1,1}^{(2)}|\leq \frac{C|{\bf x}_1|}{\delta({\bf x}_1)}\quad\mbox{and}\quad \partial_{{\bf x}_2}\bar {\bf u}_{1,1}^{(2)}=\frac{2+o(1)}{\delta({\bf x}_1)}.
\end{equation} 
Denote 
\begin{equation*}
\bar {\bf w}_{1,1}:={\bf u}_{1,1}-\bar {\bf u}_{1,1},
\end{equation*}
then we have 
\begin{align*}
\begin{cases}
\mathcal{L}_{ {\lambda}, {\mu}}\bar{\bf w}_{1,1}+\rho\omega_1^2\bar{\bf w}_{1,1}=-\mathcal{L}_{ {\lambda}, {\mu}}\bar{\bf u}_{1,1}-\rho\omega_1^2\bar{\bf u}_{1,1},  &  \bx\in D^{e},\\
{\bf w}_{1,1}=0,    & \bx\in\partial D,\\
{\bf w}_{1,1}=0,& \bx\in\partial B_R.
\end{cases}
\end{align*}
By applying a similar argument in the proof of \cite[Theorems 3.2-3.3]{LX2025} (see also, Lemma \ref{prop-wi1}), we have 
\begin{equation}\label{est-u11-}
\nabla{\bf u}_{1,1}=\nabla\bar{\bf u}_{1,1}+\Ocal(1).
\end{equation}
Thus, combining with \eqref{def-baru11} and \eqref{est-baru11}, we obtain
\begin{equation}\label{est-u11}
|\nabla{\bf u}_{1,1}(\bx)|\leq \frac{C}{\varepsilon+\kappa|{\bf x}_1|^2}+C\big(\|{\bf f}_{1,1}\|_{C^2(\partial D_1)}+\|{\bf f}_{1,2}\|_{C^2(\partial D_2)}\big),\quad\bx\in\Omega_{R_0},
\end{equation}
and
\begin{equation*}
|\nabla{\bf u}_{1,1}(0,{\bf x}_2)|\geq \frac{1}{C\varepsilon}.
\end{equation*}
Combining \eqref{est-u11-out}, \eqref{dec-u1},  \eqref{est-Du12}, and \eqref{est-u11}, we conclude 
\begin{align*}
|\nabla{\bf u}_{1}(\bx)|&\leq \frac{C}{\varepsilon+\kappa|{\bf x}_1|^2}+C\big(\|\bu_{1}\|_{L^2(D^e)}+\|{\bf g}_{1}\|_{L^{\infty}(\partial B_R)}+\|{\bf f}_{1,1}\|_{C^2(\partial D_1)}\\
&\quad+\|{\bf f}_{1,2}\|_{C^2(\partial D_2)}\big),\quad\bx\in D^e.
\end{align*}

{\bf Step 2: Estimate of $\nabla{\bu}_2$.} As in \eqref{dec-u1}, we decompose ${\bu}_2$ into 
\begin{equation*}
{\bu}_2={\bu}_{2,1}+{\bu}_{2,2},
\end{equation*}
where 
\begin{align*}
\begin{cases}
\mathcal{L}_{ {\lambda}, {\mu}}{\bf u}_{2,1}=0,  &  \bx\in D^{e},\\
{\bf u}_{2,1}={\bm{\kappa}}_2+{\bf f}_{2,1},    & \bx\in\partial D_1,\\
{\bf u}_{2,1}={\bm{\kappa}}_2+{\bf f}_{2,2},    & \bx\in\partial D_2,\\
{\bf u}_{2,1}=0,& \bx\in\partial B_R,
\end{cases}
\end{align*}
and 
\begin{align*}
\begin{cases}
\mathcal{L}_{ {\lambda}, {\mu}}{\bf u}_{2,2}=0,  &  \bx\in D^{e},\\
{\bf u}_{2,2}=0,    & \bx\in\partial D_1,\\
{\bf u}_{2,2}=0,    & \bx\in\partial D_2,\\
{\bf u}_{2,2}={\bf g}_2,& \bx\in\partial B_R.
\end{cases}
\end{align*}
Then by using similar arguments that led to Lemma \ref{prop-wi2}, we have 
\begin{equation*}
|\nabla{\bf u}_{2,2}(\bx)|\leq C\big(\|\bu_{2,2}\|_{L^2(D^e)}+\|{\bf g}_{2}\|_{L^{\infty}(\partial B_R)}\big),\quad\bx\in D^e.
\end{equation*}
In order to prove the estimate of $\nabla{\bf u}_{2,1}$, we use the auxiliary function $p$ in \eqref{def-p} to construct $\bar{\bf u}_{2,1}\in C^{2,\alpha}(\mathbb{R}^2)$, such that $\bar{\bf u}_{2,1}={\bm{\kappa}}_2+{\bf f}_{2,1}$ on $\partial D_1$, $\bar{\bf u}_{2,1}={\bm{\kappa}}_2+{\bf f}_{2,2}$ on $\partial D_2$, $\bar{\bf u}_{2,1}=0$ on $\partial B_R$, 
\begin{equation*}
\bar{\bf u}_{2,1}=p({\bm{\kappa}}_2+{\bf f}_{2,1})+(1-p)({\bm{\kappa}}_2+{\bf f}_{2,2})\quad\hbox{in}\ \Omega_{2R_0},
\end{equation*}
and $\|\bar {\bf u}_{2,1}\|_{C^{2,\alpha}(D^e\setminus\Omega_{R_0})}\leq\,C$. Then similar to \eqref{est-u11-}, we have 
\begin{equation*}
\nabla{\bf u}_{2,1}=\nabla\bar{\bf u}_{2,1}+\Ocal(1),
\end{equation*}
which gives 
\begin{align*}
|\nabla{\bf u}_{2,1}(\bx)|&\leq \frac{C|{\bf f}_{2,1}({\bf x}_1,\varepsilon/2+\mathcal{H}_1({\bf x}_1))-{\bf f}_{2,2}({\bf x}_1,-\varepsilon/2+\mathcal{H}_2({\bf x}_1))|}{\varepsilon+\kappa|{\bf x}_1|^2}\\
&\quad+C\big(\|{\bf f}_{2,1}\|_{C^2(\partial D_1)}+\|{\bf f}_{2,2}\|_{C^2(\partial D_2)}\big).
\end{align*}
Therefore, we obtain
\begin{align*}
|\nabla{\bf u}_{2}(\bx)|&\leq \frac{C|{\bf f}_{2,1}({\bf x}_1,\varepsilon/2+\mathcal{H}_1({\bf x}_1))-{\bf f}_{2,2}({\bf x}_1,-\varepsilon/2+\mathcal{H}_2({\bf x}_1))|}{\varepsilon+\kappa|{\bf x}_1|^2}\\
&\quad+C\big(\|\bu_{2}\|_{L^2(D^e)}+\|{\bf g}_{2}\|_{L^{\infty}(\partial B_R)}+\|{\bf f}_{2,1}\|_{C^2(\partial D_1)}+\|{\bf f}_{2,2}\|_{C^2(\partial D_2)}\big).
\end{align*}

{\bf Step 3: Estimates of $\nabla{\bu}_i$ with $i=3,4$.} By following similar arguments in Steps 1 and 2, we decompose ${\bu}_i$ into 
\begin{equation*}
{\bu}_i={\bu}_{i,1}+{\bu}_{i,2},
\end{equation*}
where 
\begin{align*}
\begin{cases}
\mathcal{L}_{ {\lambda}, {\mu}}{\bf u}_{i,1}=0,  &  \bx\in D^{e},\\
{\bf u}_{i,1}=e_{i,1}{\bm{\kappa}}_1+e_{i,3}{\bm{\kappa}}_3+{\bf f}_{i,1},    & \bx\in\partial D_1,\\
{\bf u}_{i,1}=e_{i,1}{\bm{\kappa}}_1+e_{i,6}{\bm{\kappa}}_3+{\bf f}_{i,2},    & \bx\in\partial D_2,\\
{\bf u}_{i,1}=0,& \bx\in\partial B_R,
\end{cases}
\end{align*}
and 
\begin{align*}
\begin{cases}
\mathcal{L}_{ {\lambda}, {\mu}}{\bf u}_{i,2}=0,  &  \bx\in D^{e},\\
{\bf u}_{i,2}=0,    & \bx\in\partial D_1,\\
{\bf u}_{i,2}=0,    & \bx\in\partial D_2,\\
{\bf u}_{i,2}={\bf g}_i,& \bx\in\partial B_R,\quad i=3,4.
\end{cases}
\end{align*}
Then we have 
\begin{equation*}
|\nabla{\bf u}_{i,2}(\bx)|\leq C\big(\|\bu_{i,2}\|_{L^2(D^e)}+\|{\bf g}_{i}\|_{L^{\infty}(\partial B_R)}\big),\quad\bx\in D^e,\quad i=3,4.
\end{equation*}
We next construct $\bar{\bf u}_{i,1}\in C^{2,\alpha}(\mathbb{R}^2)$, such that $\bar{\bf u}_{i,1}=e_{i,1}{\bm{\kappa}}_1+e_{i,3}{\bm{\kappa}}_3+{\bf f}_{i,1}$ on $\partial D_1$, $\bar{\bf u}_{i,1}=e_{i,1}{\bm{\kappa}}_1+e_{i,6}{\bm{\kappa}}_3+{\bf f}_{i,2}$ on $\partial D_2$, $\bar{\bf u}_{i,1}=0$ on $\partial B_R$, 
\begin{equation*}
\bar{\bf u}_{i,1}=p(e_{i,1}{\bm{\kappa}}_1+e_{i,3}{\bm{\kappa}}_3+{\bf f}_{i,1})+(1-p)(e_{i,1}{\bm{\kappa}}_1+e_{i,6}{\bm{\kappa}}_3+{\bf f}_{i,2})\quad\hbox{in}\ \Omega_{2R_0},
\end{equation*}
and $\|\bar {\bf u}_{i,1}\|_{C^{2,\alpha}(D^e\setminus\Omega_{R_0})}\leq\,C$. 
Direct computations give
\begin{equation*}
\partial_{{\bf x}_1}\bar {\bf u}_{i,1}^{(1)}=\frac{|{\bf x}_1|}{\delta({\bf x}_1)}o(1)+|x_1|,\quad\partial_{{\bf x}_2}\bar {\bf u}_{i,1}^{(1)}=\frac{1}{\delta({\bf x}_1)}o(1)+\Ocal(1),
\end{equation*} 
and 
\begin{equation*}
\partial_{{\bf x}_1}\bar {\bf u}_{i,1}^{(2)}=\frac{|{\bf x}_1|}{\delta({\bf x}_1)}o(1)+\Ocal(1),\quad \partial_{{\bf x}_2}\bar {\bf u}_{i,1}^{(2)}=\frac{2\bx_1+o(1)}{\delta({\bf x}_1)}.
\end{equation*} 
Similar to \eqref{est-u11-}, we have 
\begin{equation*}
\nabla{\bf u}_{i,1}=\nabla\bar{\bf u}_{i,1}+\Ocal(1),\quad i=3,4.
\end{equation*}
Thus, we obtain 
\begin{align*}
|\nabla{\bf u}_{i}(\bx)|&\leq C\left(\frac{|{\bf x}_1|}{\varepsilon+\kappa|{\bf x}_1|^2}+1\right)+C\big(\|\bu_{i}\|_{L^2(D^e)}+\|{\bf g}_{i}\|_{L^{\infty}(\partial B_R)}+\|{\bf f}_{i,1}\|_{C^2(\partial D_1)}\\
&\quad+\|{\bf f}_{i,2}\|_{C^2(\partial D_2)}\big),\quad i=3,4,\quad\bx\in D^e.
\end{align*}
When $|{\bf x}_1|\leq\varepsilon$, we get that $|\nabla{\bf u}_{i}(\bx)|$  with $i=3,4$ are bounded. When $|{\bf x}_1|>\varepsilon$, $|\nabla{\bf u}_{i}(\bx)|$ can blow up and 
\begin{equation*}
|\nabla{\bf u}_{i}({\bf x}_1,{\bf x}_2)|\geq|\partial_{{\bf x}_2}{\bf u}_{i}^{(2)}({\bf x}_1,{\bf x}_2)| \geq\frac{1}{C\sqrt{\varepsilon}},\quad |{\bf x}_1|=\sqrt{\kappa^{-1}\varepsilon},\quad i=3,4.
\end{equation*}
We complete the proof.
\end{proof}

\begin{rem}
Based on Theorem \ref{thm-eigonmode}, the behavior of $\nabla{\bf u}_{i}$ with respect to the small parameter $\varepsilon$ and spatial location can be summarized as follows: 
\begin{enumerate}[(1)]
\item For $i=1,5,6$, $\nabla{\bf u}_{i}$  blows up at the rate of $\frac{1}{\varepsilon}$ and attains its maximum at the narrowest gap between the two resonators, that is, $|{\bf x}_1|=0$. 
\item For $i=3,4$, $\nabla{\bf u}_{i}$ remains bounded when $|{\bf x}_1|\leq\varepsilon$, but blows up at the rate of $\frac{1}{\sqrt\varepsilon}$ for $|{\bf x}_1|>\varepsilon$. The maximum occurs at $|{\bf x}_1|=\sqrt{\kappa^{-1}\varepsilon}$. 
\item Whether $\nabla{\bf u}_{2}$ blows up or not depends on the difference $|{\bf f}_{2,1}({\bf x}_1,\varepsilon/2+\mathcal{H}_1({\bf x}_1))-{\bf f}_{2,2}({\bf x}_1,-\varepsilon/2+\mathcal{H}_2({\bf x}_1))|$. 
\end{enumerate}
\end{rem}

\bibliographystyle{abbrv}
\bibliography{Elastic_2D}{}

\end{document}